\documentclass[a4paper, 10pt, twoside, notitlepage]{amsart}

\usepackage[utf8]{inputenc}
\usepackage{color}
\usepackage{amsmath} 
\usepackage{amssymb} 
\usepackage{amsthm}
\usepackage{geometry}
\usepackage{graphicx}
\usepackage{esint}
\usepackage[colorlinks=true,linkcolor=blue]{hyperref}
\usepackage{booktabs}
\usepackage{todonotes}
\theoremstyle{plain}
\newtheorem{thm}{Theorem}
\newtheorem{prop}{Proposition}[section]
\newtheorem{lem}[prop]{Lemma}

\newtheorem{defi}[prop]{Definition}
\newtheorem{rmk}[prop]{Remark}

\newtheorem{example}[prop]{Example}

\newcommand {\R} {\mathbb{R}} \newcommand {\Z} {\mathbb{Z}}
\newcommand {\T} {\mathbb{T}} \newcommand {\N} {\mathbb{N}}
\newcommand {\C} {\mathbb{C}} 
\newcommand {\p} {\partial}

\newcommand {\supp} {\text{supp}}

\newcommand {\conv} {\text{conv}}

\newcommand{\Sd}{\R^{d\times d}_\textrm{sym}}
\newcommand{\Stw}{\R^{2\times 2}_\textrm{sym}}

\DeclareMathOperator{\tr}{tr}
\DeclareMathOperator{\diag}{diag}
\DeclareMathOperator{\argmin}{argmin}

\DeclareMathOperator {\dist} {dist}

\DeclareMathOperator{\F} {\mathcal{F}}

\newcommand{\st}[1]{\qquad\text{#1}\qquad}
\newcommand{\RA}{\qquad\Rightarrow\qquad}
\newcommand{\boldface}[1]{\boldsymbol{#1}}  
\newcommand{\bfk}{\boldface{k}}

\newcommand{\be}{\begin{equation}}
	\newcommand{\ee}{\end{equation}}

\newcommand{\beq}{\begin{eqnarray}}
	\newcommand{\eeq}{\end{eqnarray}}
\newcommand{\bem}{\begin{multline}}
	\newcommand{\eem}{\end{multline}}
\newcommand{\ba}{\begin{align}}
	\newcommand{\ea}{\end{align}}

\title[On a $T_3$-Structure in Geometrically Linearized Elasticity]{On a $T_3$-Structure in Geometrically Linearized Elasticity: Qualitative and Quantitative Analysis and Numerical Simulations}

\author[R.~Indergand]{Roman Indergand}
\address{ETH Zürich, Mechanik und Materialforschung,
LEE N 201, Leonhardstrasse 21, 8092 Zürich, Switzerland
}
\email{iroman@mavt.ethz.ch}
\author[D.~Kochmann]{Dennis Kochmann}
\address{ETH Zürich, Mechanik und Materialforschung,
LEE N 201, Leonhardstrasse 21, 8092 Zürich, Switzerland
}
\email{dmk@ethz.ch}
\author[A.~Rüland]{Angkana Rüland}
\address{Institut  f\"ur Angewandte Mathematik and Hausdorff Center for Mathematics, Universit\"at Bonn, Endenicher Allee 60, 53115 Bonn, Germany}
\email{rueland@uni-bonn.de}
\author[A.~Tribuzio]{Antonio Tribuzio}
\address{Institut  f\"ur Angewandte Mathematik, Universit\"at Bonn, Endenicher Allee 60, 53115 Bonn, Germany}
\email{antonio.tribuzio@uni-bonn.de}
\author[C.~Zillinger]{Christian Zillinger}
\address{Karlsruhe Institute of Technology, Englerstraße 2, 76131 Karlsruhe, Germany}
\email{christian.zillinger@kit.edu}

\begin{document}

\begin{abstract}
We study the rigidity properties of the $T_3$-structure for the symmetrized gradient from \cite{BFJK94} qualitatively, quantitatively and numerically. More precisely, we complement the \emph{flexibility} result for \emph{approximate} solutions of the associated differential inclusion which was deduced in \cite{BFJK94} by a \emph{rigidity} result on the level of \emph{exact} solutions and by a quantitative rigidity estimate and scaling result. The $T_3$-structure for the symmetrized gradient from \cite{BFJK94} can hence be regarded as a symmetrized gradient analogue of the Tartar square for the gradient.
As such a structure cannot exist in $\R^{2\times 2}_{sym}$ the example from \cite{BFJK94} is in this sense minimal. We complement our theoretical findings with numerical simulations of the resulting microstructure.
\end{abstract}

\maketitle

\section{Introduction}

Motivated by quantitative rigidity results for singular perturbation problems
for the Tartar square and by problems from the materials sciences, it is the
purpose of this article to study a particular $T_3$-structure which had been introduced in \cite{BFJK94} for the \emph{symmetrized gradient}. More precisely, we consider qualitative and quantitative properties of microstructures associated with the differential inclusion
\begin{align}
\label{eq:BFJK}
e( u) \in K_3:=\left\{ e^{(1)}, e^{(2)}, e^{(3)}
 \right\} \subset \R^{3\times 3}_{sym}
 \mbox{ a.e. in } \Omega ,
\end{align}
for a general $\Omega\subset\R^3$ open, bounded and connected, and with
\begin{align*}
e^{(1)} = \begin{pmatrix}
\eta_3 & 0 & 0 \\
0 & \eta_1 & 0\\
0 & 0 & \eta_2
\end{pmatrix}, \
e^{(2)} = \begin{pmatrix}
\eta_2 & 0 & 0 \\
0 & \eta_3 & 0\\
0 & 0 & \eta_1
\end{pmatrix}, \
e^{(3)} = \begin{pmatrix}
\eta_1 & 0 & 0 \\
0 & \eta_2 & 0\\
0 & 0 & \eta_3
\end{pmatrix}.
\end{align*}
Here $e( u):= \frac{1}{2}(\nabla u + (\nabla u)^t)$ denotes the symmetrized gradient and the entries
 $\eta_j \in \R$ satisfy the conditions of \cite{BFJK94}:
\begin{align}
\label{eq:elliptic}
\eta_2< \eta_1 < \eta_3, \ \eta_2 + \eta_3 > 2\eta_1.
\end{align}
We recall that this differential inclusion was introduced in \cite{BFJK94}.
Here, as in the classical $m$-well problem, the condition \eqref{eq:elliptic}
can be viewed as an ellipticity condition. It excludes the presence of
symmetrized rank-one connections between the wells and thus of (trivial)
simple-laminate microstructures, i.e., certain one-dimensional solutions. It is
one of the objectives of the present article to highlight the parallel
between the Tartar square and the $T_3$-structure from \cite{BFJK94} qualitatively and quantitatively. 
In what follows below, we thus study both the structure of \emph{exact} solutions to the associated differential inclusion problem \eqref{eq:BFJK} and energetic quantifications of these. As a key novel difficulty compared to the setting of the Tartar square, in our investigation of the rigidity and flexibility properties of \eqref{eq:BFJK} we are confronted with a situation in which \emph{gauges} are present. These arise in the form of the infinitesimal frame indifference encoded in the gauge group $\mbox{skew}(3)$ and provide novel challenges which have to be addressed in our analysis.

\subsection{The Tartar square and its geometrically linearized analogues}
In order to put the $T_3$-structure from above into its context, we recall the \emph{Tartar square}. This is a well-known set $K_4 \subset \R^{2\times 2}_{diag}$ consisting of four diagonal matrices without rank-one connections. It plays an important role in various inner-mathematical and materials science applications and has been discovered in various contexts \cite{T93,Sch75,CT93,NM91,S93,FS08,MS,MS03} (see also \cite{M1} for further remarks and references). In particular, it is the simplest example of a set for which a first (weak) dichotomy between rigidity and flexibility occurs in the context of the $m$-well problem for the \emph{gradient} in spite of the absence of rank-one connections:
\begin{itemize}
\item On the one hand, for $\Omega \subset \R^2$ open, bounded and connected, \emph{exact} solutions $u \in W^{1,\infty}(\Omega, \R^2)$ to the differential inclusion $\nabla u \in K_4$ a.e. in $\Omega$ are \emph{rigid} in the sense that only affine solutions exist to this differential inclusion problem. 
\item On the other hand, \emph{approximate} solutions are flexible in the sense that there exists a sequence $(u_k)_{k\in \N} \subset W^{1,\infty}(\Omega, \R^2)$ such that
$\dist(\nabla u_k, K_4) \rightarrow 0$ in measure but $(\nabla u_k)_{k \in \N}$ does \emph{not} converge to a constant matrix in measure along any subsequence.
\end{itemize}
This is a rather striking behaviour as it is known that for any set $K_m$ consisting of $m \in\{1,2,3\}$ matrices which are pairwise not rank-one connected, \emph{both} the exact and the approximate problems are rigid. Thus, for $m=4$ a first loss of rigidity occurs in that rigidity of approximate solutions may fail (for specific choices of the sets $K_m$). By a result due to Chleb\'ik and Kirchheim \cite{CK00} the Tartar square, in a precise sense, is a minimal example of this. For $m\geq 5$ the problem may become completely flexible in that there exists a set $K_5$ such that \emph{both} the approximate and exact problems loose rigidity \cite{K1,FS17}. We refer to the articles \cite{B90,Sverak,CK00,P10} for further references on the $m$-well problem for the gradient, as well as to the books and lecture notes \cite{D,M1,Ri18} for an overview of the seminal literature on this.

Motivated by this dichotomy for the gradient and by problems from the materials sciences which inherently include (infinitesimal) frame indifference and thus associated \emph{gauges}, in the present article we study qualitative and quantitative properties of the (almost) physical system from \cite{BFJK94} for the \emph{symmetrized gradient}. Since any differential inclusion for the symmetrized gradient can also be interpreted as a differential inclusion for the gradient (involving affine vector spaces), the results on the loss of rigidity for the gradient can be embedded into this context. In particular, for the symmetrized gradient rigidity for approximate solutions can thus at best hold for $m\in \{1,2,3\}$. An analogous observation holds for exact solutions. In this context, it is well-known that rigidity for the differential inclusion problem for the symmetrized gradient holds for $m=1,2$ (c.f. the discussion in \cite{BFJK94}, e.g., Theorem 1.4). 

In contrast to the $m$-well problem for the gradient, however, for the symmetrized gradient the first loss of rigidity -- in the sense of flexibility for approximate solutions -- already occurs at $m=3$. 
Indeed, the article \cite{BFJK94} gives an (almost) physical example of a set $K_3 \subset \R^{3\times 3}$ such that \emph{approximate} solutions to the differential inclusion for the \emph{symmetrized} gradient are flexible. Just as the Tartar square, this set of matrices carries the structure of a (symmetrized) $T_N$-set (see Definition \ref{def:T3}, in the case $N=3$, below and also \cite{KMS03} for the more general case).
It is the main goal of this article to qualitatively, quantitatively and numerically study the associated rigidity properties of this set and to investigate the role of the presence of the non-trivial gauge group $\mbox{skew}(3)$.

\subsection{Main results in the qualitative setting}

The objective of the first part of this article is to complement the \emph{flexibility} properties for approximate solutions from \cite{BFJK94} by investigating the \emph{qualitative rigidity properties} of exact solutions to the differential inclusion \eqref{eq:BFJK}. To this end, we combine the compatibility conditions with the associated structure of the set in matrix space. Our first main result reads as follows.

\begin{thm}[Rigidity of the exact inclusion]
\label{thm:rigidity0}
Let $\Omega \subset \R^3$ be a bounded, simply connected Lipschitz set.
Let $e_{11}, e_{22}, e_{33}: \Omega \rightarrow \R$ be the diagonal
components of the strain matrices $e( u)$ solving \eqref{eq:BFJK} and let \eqref{eq:elliptic} hold. Then, all functions $e_{jj}$, $j\in\{1,2,3\}$, are constant.
\end{thm}

As a consequence, the example from \cite{BFJK94} can indeed be considered as one of the simplest and also one of the most fundamental analogues of the Tartar square for the symmetrized gradient. Indeed, it is an example for which there is a first loss of rigidity in that approximate solutions are flexible while exact solutions are rigid. 
In this regard, in \cite{BFJK94} it has been also proved that two-dimensional Young measures that are limits of strains and are supported on a discrete set of incompatible symmetric matrices are Dirac deltas.
We translate this result into our context in Proposition \ref{prop:2d} and thus recall the absence of $T_3$-structures for two-dimensional strains. This justifies how the set in \eqref{eq:BFJK} may be interpreted as the simplest analogue of the Tartar square.

We emphasize that the combination of Proposition \ref{prop:2d} and Theorem
\ref{thm:rigidity0} illustrates that, in contrast to the gradient case, for
the symmetrized gradient there is a difference between the two- and three-dimensional settings in terms of rigidity and (approximate) flexibility
properties. In particular, our discussion suggests that a two-dimensional
reduction as used in \cite{CK00} may not exist for the symmetrized gradient --
at least if it is such that also approximate flexibility is preserved by it --
and thus necessitates a three-dimensional rigidity analysis as carried out
below. In addition to this,  our three-dimensional analysis is also of
independent interest as it provides a strategy towards a quantitative
understanding of $T_3$-structures for the symmetrized gradient, in parallel to
the analysis from \cite{RT22}. We turn to this in the second part of this article.

Moreover, we highlight that the example from \cite{BFJK94} (and thus of the inclusion \eqref{eq:BFJK}) is of particular relevance in the context of martensitic phase transitions since it corresponds to an (almost) physical setting. Indeed, already in \cite{BFJK94} it is pointed out that this example is a type of cubic-to-orthorhombic transformation which however is not observed experimentally. Moreover, the inclusion for the symmetrized gradient takes into account (infinitesimal) frame indifference -- an important feature of martensitic phase transformations --, while the classical $m$-well problem does not.

We emphasize that our analysis is strongly tied to the specific structure of the set in \eqref{eq:BFJK}. It remains an interesting open problem to further study the analogue of the $T_m$ problem for the symmetrized gradient more systematically. We hope to return to this in future projects.

\subsection{Main results in the quantitative setting}
In the second part of this article, we turn to the \emph{quantitative rigidity properties} of the symmetrized $T_3$-structure from \cite{BFJK94}. In order to formulate the precise set-up and our results, we introduce elastic and surface energies. Considering the torus of side length one in each direction, $\Omega = \T^3$, in what follows below, the elastic energy is modelled by a quadratic Hooke's law
\begin{align}
\label{eq:elast}
E_{el}(u, \chi) := \int\limits_{\Omega}|e(u)- \chi|^2 dx,\
\end{align}
where
\begin{align}
\label{eq:chi}
\chi = \diag(\eta_3 \chi_1 + \eta_2 \chi_2 + \eta_1 \chi_3, \eta_1 \chi_1 + \eta_3 \chi_2 + \eta_2 \chi_3, \eta_2 \chi_1 + \eta_1 \chi_2 + \eta_3 \chi_3).
\end{align}
Here the functions $\chi_j : \Omega \rightarrow \{0,1\}$ with $\chi_1 + \chi_2 + \chi_3 = 1 \mbox{ a.e.~in }\Omega$ are viewed as phase indicators, keeping track of the variant of martensite at a respective material point. 
The elastic energy thus measures the deviation from being an exactly stress-free solution to the differential inclusion \eqref{eq:BFJK}. The surface energy in turn quantifies the presence of high oscillations and yields control on the transition between the wells:
\begin{align}
\label{eq:surf}
E_{surf}(\chi):= \sum\limits_{j=1}^{3} \int\limits_{\Omega}|D \chi_j|.
\end{align}
Here $\int\limits_{\Omega}|D \chi_j|$ denotes the BV semi-norm of the function $\chi_j$.
We combine both contributions into the full energy 
\begin{align}
\label{eq:energy_total}
E_{\epsilon}(u,\chi):= E_{el}(u,\chi) + \epsilon E_{surf}(\chi).
\end{align}
In this context, the parameter $\epsilon>0$ should be considered as a small parameter; we are particularly interested in the limit $\epsilon \rightarrow 0$.
We consider the total energy as a regularization of the elastic contribution and seek to investigate its scaling properties in the singular perturbation parameter $\epsilon>0$. We view this combination of elastic and surface energies as encoding the complexity of the resulting microstructure and, in particular, of approximate solutions.

As in the setting of the Tartar square, the matrices $e^{(j)}$, $j\in \{1,2,3\}$, are not symmetrized rank-one connected. However, as is characteristic of $T_m$ structures, there are three auxiliary matrices,  
\begin{align*}
J_1:= \begin{pmatrix}
\kappa & 0 & 0 \\
0 & \eta_1 & 0\\
0 & 0 & \eta_1
\end{pmatrix}, \
J_2:= \begin{pmatrix}
\eta_1 & 0 & 0 \\
0 & \kappa & 0\\
0 & 0 & \eta_1
\end{pmatrix},\
J_3:= \begin{pmatrix}
\eta_1 & 0 & 0 \\
0 & \eta_1 & 0\\
0 & 0 & \kappa
\end{pmatrix},
\end{align*}
with $\kappa = \eta_2+ \eta_3 - \eta_1$ which are pairwise symmetrized rank-one connected and which are each symmetrized rank-one connected to one of the wells $e^{(j)}$, $j\in\{1,2,3\}$ (see Lemma \ref{lem:BFJK94} and the discussion around Lemma 3.3 in \cite{BFJK94}). We refer to Figure \ref{fig:T3} for an illustration of this relation. The matrices $J_1,J_2, J_3$ play a fundamental role in our quantitative analysis below. Here the chosen labelling of the matrices $J_j$ will be particularly convenient. In turn, it is this ordering, which is also at the origin of the slightly unconventional labelling of the strain matrices $e^{(j)}$ in \eqref{eq:BFJK} (which deviates from the ordering from \cite{BFJK94}).

\begin{figure}[t]
\includegraphics[width = 0.5 \textwidth]{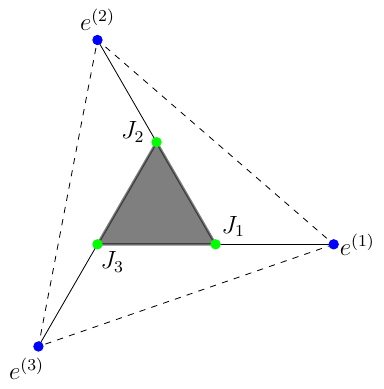}
\caption{A schematic illustration of a $T_3$-structure as in \eqref{eq:BFJK}. The inner triangle is spanned by the auxiliary matrices $J_{1}, J_{2}, J_3$ (green points). These are pairwise symmetrized rank-one connected as indicated by the lines between them. The outer (rotated) triangle consists of the wells $e^{(1)}, e^{(2)}, e^{(3)}$. As indicated by the dashed lines, these are not symmetrized rank-one connected. However, as illustrated by the straight full lines, each of these is symmetrized rank-one connected to one of the three auxiliary matrices $J_1, J_2, J_3$.
}
\label{fig:T3}
\end{figure}

With the help of these definitions and auxiliary observations, we can formulate our first quantitative result on the $T_3$-structure from \eqref{eq:BFJK}.

\begin{thm}[Rigidity]
\label{thm:rigidity}
Let $\bar{e} \in \R^{3\times 3}_{sym}$ be such that $\bar{e}\in \conv\{J_1,J_2,J_3\}\cup \bigcup\limits_{j=1}^3 [J_j, e^{(j)}]$.
Let $E_{\epsilon}(u,\chi)$ be as in \eqref{eq:energy_total} with $\Omega = \T^3$ and assume that $u \in \mathcal{A}_{\overline{e}}$ with
\begin{align}
\label{eq:admissible}
\mathcal{A}_{\bar{e}}:= \left\{ u\in H^1_{loc}(\R^3, \R^3): \ \nabla u \in L^2(\T^3, \R^{3\times 3}), \ \langle e( u) \rangle  = \bar{e}   \right\},
\end{align}
where $\langle e( u )\rangle := \int\limits_{\T^3} e( u) dx$ denotes the average of $e( u)$ on $\T^3$.
Then for all $\nu \in (0,1)$ there exist $\epsilon_0>0$ and $c_{\nu}>0$ such that for all $\epsilon \in (0,\epsilon_0)$ we have that
\begin{align*}
\|e(u)- \bar{e}\|^4_{L^2(\T^3)} \leq \exp(c_{\nu}(|\log(\epsilon)|^{\frac{1}{2} + \nu}+1)) E_{\epsilon}(u,\chi),
\end{align*}
for every $\chi\in BV(\T^3;K_3)$
\end{thm}

Let us comment on this result: As for the Tartar square, Theorem \ref{thm:rigidity} quantifies both the \emph{rigidity} properties (on the level of exactly stress-free solutions) and the \emph{flexibility} (on the level of approximate solutions) of the $T_3$-structure \eqref{eq:BFJK}. Indeed, on the one hand, the bound measuring the deviation from the constant state $\bar{e}$ encodes the fact that stress-free solutions are necessarily constant and thus rigid. On the other hand, the scaling law of the form $\exp(c_{\nu}|\log(\epsilon)|^{\frac{1}{2} + \nu})$ encodes the \emph{flexibility} of having nontrivial approximate solutions to the differential inclusion \eqref{eq:BFJK}. The slow, subalgebraic, superlogarithmic behaviour as $\epsilon \rightarrow 0$ however gives a strong indication that this requires highly complex microstructures -- laminates of infinite order -- which match the observations found in \cite{BFJK94} (see the discussion of the numerical results in Section \ref{sec:numerics} below).

\begin{rmk}
\label{rmk:loglog}
We remark that instead of phrasing the results in Theorems \ref{thm:rigidity}, \ref{thm:scaling} and \ref{thm:rigidity_Dirichlet} (in Section \ref{sec:Dirichlet}) with the (arbitrarily small) algebraic losses encoded in the power $\nu \in (0,1)$, it would also be possible to formulate this with (even smaller) $\log(|\log(\epsilon)|)$-type losses. Indeed, our proof in Section \ref{sec:thms_proof} is phrased in this way. As we expect that the optimal behaviour does not require any form of these losses while our current method of proof cannot avoid these losses (due to the use of Calder\'on Zygmund estimates) we did not optimize these dependences further.
\end{rmk}

The bounds of Theorem \ref{thm:rigidity} further entail a resulting scaling law.
This law displays a subalgebraic, superlogarithmic behaviour as in the Tartar setting.

\begin{thm}[Scaling law]
\label{thm:scaling}
Let $E_{\epsilon}(u,\chi)$ be as above with $\Omega = \T^3$.
Let $\bar{e} \in \R^{3\times 3}_{sym}$ be such that $\bar{e}\in \conv\{J_1,J_2,J_3\}\cup \bigcup\limits_{j=1}^3 [J_j, e^{(j)}]$
and let 
\begin{align*}
E_{\epsilon}(\bar{e}) := \inf\limits_{\chi \in BV(\Omega, \{e^{(1)}, e^{(2)}, e^{(3)}\})} \inf\limits_{u \in \mathcal{A}_{\overline{e}}} E_{\epsilon}(u,\chi),
\end{align*}
where $\mathcal{A}_{\overline{e}}$ is as in \eqref{eq:admissible}.
Then for all $\nu\in (0,1)$ there exist $\epsilon_0>0, c_{\nu}>0, C_1, C_2,C>0$ and functions $r_1:(0,\epsilon_0) \rightarrow \R_+, r_1(t) = C_1 \exp(-c_{\nu}|\log(t)|^{\frac{1}{2} + \nu})$, $r_2:(0,\epsilon_0) \rightarrow \R_+, r_2(t) = C_2 \exp(-C|\log(t)|^{\frac{1}{2} })$ such that for all $\epsilon \in (0,\epsilon_0)$ we have that
\begin{align*}
r_1(\epsilon)\leq E_{\epsilon}(\bar{e}) \leq r_2(\epsilon).
\end{align*}
\end{thm}

As indicated above and in parallel to the setting for the Tartar square, the slow decay behaviour in $\epsilon \downarrow 0$ quantifies the high complexity of the infinite order laminates forming the macroscopically visible microstructure in the $T_3$-structures from above.

\begin{rmk}[Dirichlet vs periodic boundary data for the symmetrized gradient]
\label{rmk:Dirichlet}
We remark that variants of Theorems \ref{thm:rigidity} and \ref{thm:scaling} remain valid in general bounded Lipschitz domains in which affine Dirichlet data are prescribed and with symmetrized gradient in a suitable subset of $\conv\{J_1,J_2,J_3\}\cup \bigcup\limits_{j=1}^3 [J_j, e^{(j)}]$. We outline the changes and the slight restriction of the possible set of boundary data in Section \ref{sec:Dirichlet}.
\end{rmk}

\subsection{Main challenges and ideas associated with the proofs of Theorems \ref{thm:rigidity0}-\ref{thm:scaling}.}
Let us comment on the main ideas and challenges associated with deducing the results of Theorems \ref{thm:rigidity0}-\ref{thm:scaling}. 

While inspired by the recent articles \cite{RT22, RT22a, RT23} the present
setting is (even on an algebraic level) substantially more complicated due to the presence of the gauge group $\mbox{skew}(3)$. This is
directly reflected in the structure of stress-free strains of diagonal form: By
the results of \cite{DM2}, under a constant trace constraint, these can be
characterized as the sum of six planar waves (see Section \ref{sec:struc}). In
contrast to the setting of the Tartar square, it is thus \emph{not} immediately
obvious that the $T_3$-structure from \eqref{eq:BFJK} is rigid in the sense that
only affine solutions exist to the differential inclusion. More precisely, a key
step in the proof of Theorem \ref{thm:rigidity0} consists in excluding possibly
complex interaction of the six planar waves into which any diagonal, constant
trace, exactly stress-free deformation can be decomposed and which could -- a priori -- result in complex microstructures. As in the Tartar setting, the ellipticity condition \eqref{eq:elliptic} implies a ``determinedness'' property, i.e., the fact that each diagonal component of one of the matrices \eqref{eq:BFJK} determines all other components. However, we emphasize that as proved by the counterexample of Kirchheim-Preiss \cite{K1,P10,FS17}, in general, ellipticity alone may not suffice to exclude complicated microstructures. In particular, while our result provides a first rigidity result for diagonal $T_3$-structures for the symmetrized gradient, it, for instance, does not immediately imply an analogous result for other, non-diagonal $T_3$-structures (such as the ones in \cite{CS13}) in $\R^{3\times 3}_{sym}$ for the symmetrized gradient. We refer to the discussion in Section \ref{sec:lit} below. We further highlight that the example from \eqref{eq:BFJK} provides a  minimal example for a $T_3$-structure for the symmetrized gradient. As is probably well-known to experts and as recalled in Section \ref{sec:2D} below, in two-dimensional, symmetric matrices such structures do \emph{not} exist. The $T_3$-structure from \cite{BFJK94} and \eqref{eq:BFJK} hence plays a central role in the further study of $T_3$-structures for the symmetrized gradient.

The outlined difficulties in the derivation of the rigidity result of Theorem \ref{thm:rigidity0} are also directly reflected in the proof of the quantitative statements of Theorems \ref{thm:rigidity}, \ref{thm:scaling}.
We emphasize that while there are seminal results on quantitative rigidity
properties for three-dimensional symmetrized elasticity (see \cite{CO, CO1}), new
difficulties emerge in the context of quantifying the rigidity properties of the $T_3$-structure from above. In \cite{RT23}, which was strongly inspired by the methods from \cite{CO1}, these methods could still be adapted to deduce scaling laws in the presence of Dirichlet data consisting of second order laminates. In a sense, we view the methods from \cite{RT23} (building on \cite{CO, CO1}) as tailored towards second order laminates. Contrary to the setting of \cite{RT23} in the setting of the $T_3$-structure we expect the presence of laminates of \emph{infinite order} and all quantitative arguments have to reflect this.
Hence, guided by the argument for Theorem \ref{thm:rigidity0}, we mimic the qualitative argument as far as possible and exploit the ellipticity and ``determinedness structure'' of the wells. This is strongly inspired by the work \cite{RT22}. However, due to the presence of the gauge group $\mbox{skew}(3)$ and, as a consequence, the six planar waves in the characterization of exactly stress-free strains, a ``naive'' version of the argument of \cite{RT22} does not result in the desired iterative gain of control of the wave decomposition.
Thus, instead, we carefully extract the most relevant directions in the plane
wave decomposition and only then apply the determinedness result (in parallel to
the stress-free setting). In a precise sense, this breaks the symmetries of the problem (consisting of all six possible waves). Only after having isolated the central relevant waves it is then possible to pursue the outlined iterative support improvement strategy. Such a breaking of symmetry property (together with the determinedness property) constitutes part of the nonlinear argument. In a qualitative context this has first been introduced in \cite{DM2} and has subsequently been used in different situations involving laminates of order at most two in \cite{CO, CO1, R16, RS23}.
We expect that such a strategy is very robust and can also be applied in settings in which the energy wells involve an arbitrarily large number of finite laminates, similarly as in \cite{RT22a}.
We plan to explore this in future work.

\subsection{Main results on the numerical simulation of the microstructure}
\label{sec:numerics}
In order to compare the above results with experimental settings, we have also carried out numerical simulations of possible microstructures emerging at low energy.
To this end, we focused on the minimization problem 
\begin{align}\label{eq:Emin}
E_{\epsilon}(\bar{e}) := \inf\limits_{\chi \in BV(\Omega, \{e^{(1)}, e^{(2)}, e^{(3)}\})} \inf\limits_{u \in \mathcal{A}_{\bar{e}}} E_{\epsilon}(u,\chi),
\end{align}
from Theorem \ref{thm:scaling} with the two specifically chosen average strain conditions
\begin{align*}
\bar{e}:= J_1:= \diag(\eta_1, \eta_1, \kappa)
\end{align*} 
and
\begin{align*}
\bar{e}:=B:= \diag\left( \frac{2}{3} \eta_1 + \frac{1}{3} \kappa,  \frac{2}{3} \eta_1 + \frac{1}{3} \kappa,  \frac{2}{3} \eta_1 + \frac{1}{3} \kappa \right)=\tfrac{1}{3}J_1 + \tfrac{1}{3}J_2+\tfrac{1}{3}J_3.
\end{align*}
For the simulations we made the specific choices $\eta_1 = 0.03$, $\eta_2 = 0.01$, $\eta_3 = 0.06$ and thus, in particular, $\kappa = 0.04$. We note that these satisfy the ellipticity conditions, so that we are thus in a setting in which the results from the previous sections are valid.
\begin{figure}[t]
	\centering
    \includegraphics[width=0.9\textwidth]{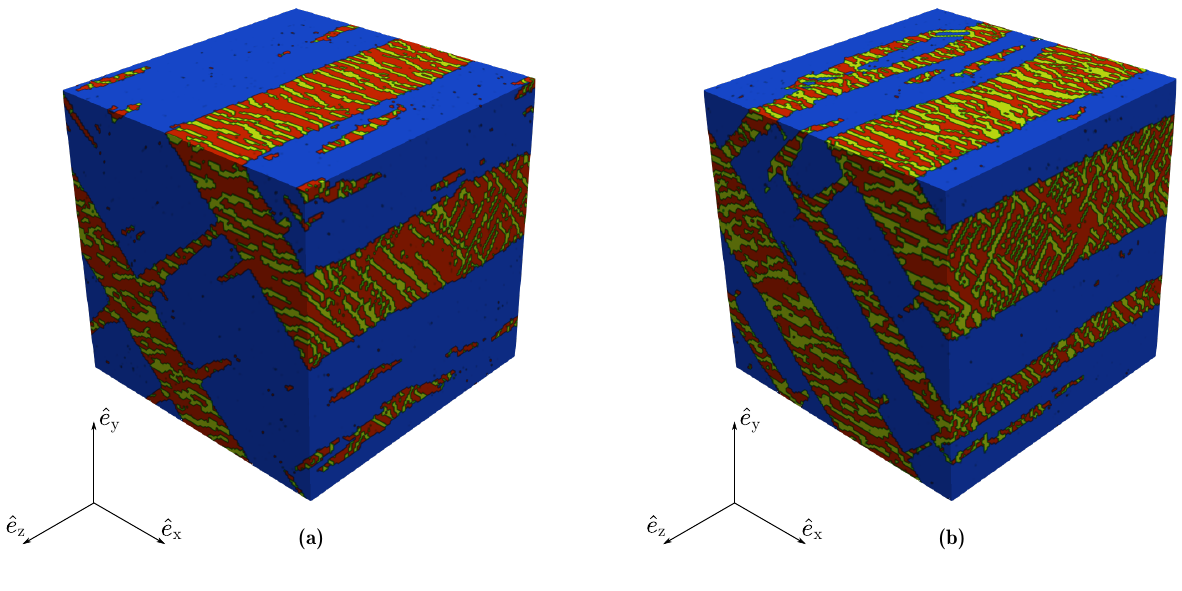}
    \caption{Numerical results of $T_3$-microstructures computed in a spatially resolved framework over a representative volume element by imposing macroscopic strains: (a) $\bar{e}=J_1$ and (b) $\bar{e}=B$. The colors, blue, red, and yellow, indicate the corresponding fraction of each phase, $\chi_1$, $\chi_2$, and $\chi_3$, respectively.}
	\label{fig:numRes4MatrixJandB}
\end{figure}

\begin{figure}[t]
	\centering
    \includegraphics[width=0.5\textwidth]{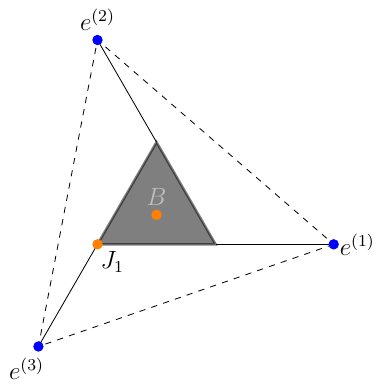}
    \caption{A schematic illustration of the boundary conditions $J_1$ and $B$
      used in Figure \ref{fig:numRes4MatrixJandB}.}
	\label{fig:JandB}
\end{figure}

Instead of minimizing \eqref{eq:Emin} directly, our numerical simulations solve the linear elastic equations of static equilibrium, linear momentum balance, over a periodic domain $\Omega$, utilizing a Fourier-based spectral discretization of the displacement field in three dimensions \cite{Moulinec1998}. By introducing the fractions of the three wells in a continuous fashion but penalizing those to remain close to $0$ or $1$, we can associate each material point in the equilibrated configuration with one of the three wells, while locally admitting small, elastic perturbations around each well \cite{KumarEtAl2020}. Although considerable numerical expenses in three dimensions prevent us from going to high spatial resolution, the chosen resolution is sufficient to observe the emergence of complex microstructures. (Details on the numerical scheme are provided in Part~\ref{part:3}.)

The numerical results, which illustrate the distribution of the three variants within the periodic unit cell, indeed reflect the typical $T_N$ behaviour. We obtain bands of laminates within laminates, see Figure~\ref{fig:numRes4MatrixJandB}. In the simulations these are observed up to second order. Indeed, as in our construction from Section \ref{sec:upper} in both final states, we observe a first, outer laminate (in blue) and a second, inner laminate (with inner bands of red-yellow). The associated lamination directions match the predicted ones from the upper bound construction in Section~\ref{sec:upper}. We expect that the further anticipated higher-order laminates in the simulations are suppressed 
by the limited spatial resolution (a consequence of the computational costs in these 3D simulations, see Remark \ref{rmk:exp}).

\subsection{Relation to the literature}
\label{sec:lit}
In concluding the introduction, we highlight further connections of our results to the literature. Firstly, we point out that the $T_3$-structure in \eqref{eq:BFJK} does \emph{not} constitute the only possible $T_3$-structure in $\R^{3\times 3}_{sym}$. As pointed out in \cite{BFJK94} a perturbation argument allows to construct $T_3$-structures associated with the cubic-to-monoclinic phase transformation from it (and even ones with full $SO(3)$ invariance). Such structures have also been found and have been systematically studied in \cite{CS13, CS15}. In general, these structures cannot be immediately reduced to each other since, for instance, the structures in \cite{CS13} are not simultaneously diagonalizable. A full classification of $T_3$-structures for the symmetrized gradient in $\R^{3\times 3}_{sym}$ remains an open problem.

Secondly, we emphasize that the generalization of the $m$-well problem to more general differential inclusions has been studied recently. We refer to \cite{GN04, PP04} for problems for the divergence operator, to \cite{TS23} for the study of flexibility of exact solutions associated with $T_4$-structures for general linear differential operators and to \cite{DPR20,RRT22} for generalized two-well problems and a quantitative analysis of the $T_3$-structure from \cite{GN04, PP04}. While \cite{GN04, PP04,RRT22} also deal with $T_3$-structures in three dimensions (with the divergence operator as its annihilator), we stress the major difference between the present article and those results: With a first order operator as its annihilator the examples from \cite{GN04, PP04,RRT22} are still substantially closer to the original Tartar square. For instance, all compatibility conditions involve first order operators and the decomposition of the stress-free states is much less involved. The differential inclusion for the symmetrized gradient however involves a \emph{second order operator} as its annihilator, which already on the algebraic level leads to substantial complications. We refer to \cite{RRT22,RRTT23} for a systematic overview on the role of the structure of the annihilator of the differential inclusion.

Finally, we point out further results on the (quantitative) dichotomy between
rigidity and flexibility and scaling laws used for the study of the complexity
of microstructures in shape-memory alloys and related models. For a
non-exhaustive list, we refer to \cite{BM23, CO1,CO, KKO13,RTZ19, RZZ16, RT22, GRTZ24,GZ24,
  RT22, RT22a, CDPRZZ20, CKZ17, KW14, KW16, PW22, DPR20, KLLR19a, KLLR19b, C99, W97,R22, DM1, DM2, RS23, S21,S21a} and the references therein.

\subsection{Outline of the article}
The remainder of this note is organized as follows: Part \ref{part:1} deals with the qualitative rigidity properties of the differential inclusion \eqref{eq:BFJK}. More precisely, in Section \ref{sec:pre} we recall the notions of compatibility and the definition of a $T_3$-structure. Next, in Section \ref{sec:2D} we
give a proof of the absence of two-dimensional $T_3$-structures
and in Section \ref{sec:BFJK} we present the proof of Theorem \ref{thm:rigidity0}. 

Part \ref{part:2} considers the energetically quantified setting and derives the results in Theorems \ref{thm:rigidity}, \ref{thm:scaling}. This is split into three main sections: In Section \ref{sec:pre2}, we prove and recall central energy estimates for the energies introduced above. Building on this, in Section \ref{sec:lower} we devise an iterative procedure of studying the phase-space contributions in which the elastic energy is not coercive. Here we exploit the elliptic structure of the wells. Based on this, we deduce Theorems \ref{thm:rigidity} and \ref{thm:scaling} in Section \ref{sec:thms_proof}. In Section \ref{sec:upper} we complement this by presenting the arguments for the upper bound in Theorem \ref{thm:scaling}.

Last but not least, in Part \ref{part:3} we discuss the numerical simulations of the microstructures associated with the $T_3$ differential inclusion \eqref{eq:BFJK}.

\part{Qualitative Results}
\label{part:1}

\section{Preliminaries and main definitions}
\label{sec:pre}

In this section, we recall a number of notions which will be convenient in what
follows. If not specified otherwise, in our qualitative discussion we will, in particular, always consider
$\Omega\subset\R^d$ a simply connected and open set with Lipschitz boundary.

\subsection{Strain compatibility and \texorpdfstring{$T_3$}{T3}-structures}

We first recall the concept of \emph{strain compatibility} of two symmetric matrices.
\begin{defi}\label{def:comp}
We say that two distinct matrices $A, B\in\Sd$ are \emph{(strain) compatible} if there exist two vectors $a,b\in\R^d\setminus \{0\}$ such that
\begin{align}
  \label{eq:comp}
A-B=\frac{1}{2}(a\otimes b+b\otimes a).
\end{align}
Otherwise we say that they are \emph{incompatible}.
The directions parallel to $a$ and $b$ above are called \emph{twin directions}.
\end{defi}

In order to introduce $T_3$-structures we use the definition given in \cite{CS13} generalized to all $d\times d$ symmetric matrices and without the trace constraint from \cite{CS13} (see also \cite{BFJK94}).
\begin{defi}[Definition 4.1 \cite{CS13}]\label{def:T3}
Three matrices $A_1, A_2, A_3\in\Sd$ form a $T_3$\emph{-structure} (for the symmetrized gradient) if
\begin{itemize}
\item[(i)] they are pairwise incompatible,
\item[(ii)] there exist $\tilde{J}_1,\tilde{J}_2,\tilde{J}_3\in\Sd$ and $\lambda_1,\lambda_2,\lambda_3\in(0,1)$ such that
\[
\tilde{J}_1=\lambda_1 A_2+(1-\lambda_1)A_3, \quad
\tilde{J}_2=\lambda_2 A_3+(1-\lambda_2)A_1, \quad
\tilde{J}_3=\lambda_3 A_1+(1-\lambda_3)A_2,
\]
and $\tilde{J}_j$ and $A_j$ are compatible, $j=1,2,3$.
\end{itemize}
\end{defi}
The matrices $\tilde{J}_{j}$ and their relation to $e^{(j)}$ and $J_j$ are shown
in Figure \ref{fig:Jtilde}.

\begin{figure}
  \centering
  \includegraphics[width=0.7\linewidth]{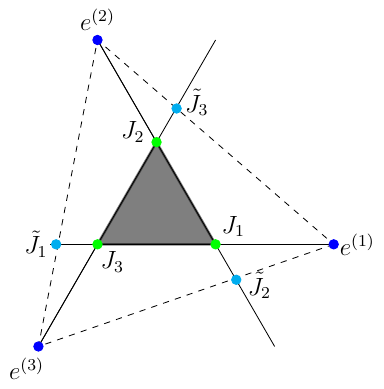}
  \caption{A schematic illustration of a $T_3$-structure. The matrices
    $\tilde{J}_j$ of Definition \ref{def:T3} are also symmetrized rank-one
    connected to the matrices $J_{i}$ of Figure \ref{fig:T3}.}
     \label{fig:Jtilde}
\end{figure}

\begin{rmk}
It has been shown in \cite{BFJK94} that, regardless of the dimension, there are
no $T_3$-structures for gradients, that is if the notion of (strain)
compatibility  is replaced by rank-1-compatibility in Definition \ref{def:T3}.
\end{rmk}

As in the setting of gradients, the notion of $T_3$-structures for the
symmetrized gradient can be extended to the notion of a more general $T_N$-structure. We refer to \cite{KMS03} for  a discussion of this extension for the gradient.

\section{Absence of \texorpdfstring{$T_3$}{T3}-structures in two dimensions}
\label{sec:2D}

In order to identify the ``simplest'' interesting situation in which $T_3$-structures for the symmetrized gradient may occur, we first recall that 
there do not exist $T_3$-structures in $\R^{2\times 2}_{sym}$.

\begin{prop}
\label{prop:2d}
There exist no $T_3$-structures for the symmetrized gradient in $\Stw$.
\end{prop}

We emphasize that this result is already present in \cite[Theorem 1.4(a) and Theorem 2.3]{BFJK94} in the more general context for Young measures and that we do not claim any novelty for this result. We only present a short self-contained argument for the rigidity result of Proposition \ref{prop:2d} for the exact differential inclusion for completeness.

We begin by further recalling that
the compatibility condition for two symmetric matrices in two dimensions reduces to the non-positivity (as bilinear forms) of their difference.
This fact is well-known in the literature, see for instance \cite[Lemma
  4.1]{K91} or \cite[Lemma 1.2]{BFJK94} and the references therein.
Here we give an elementary proof for the sake of completeness.

\begin{lem}\label{lem:2x2comp}
Given two matrices $A,B\in\Stw$, they are compatible if and only if $\det(A-B)\le0$
\end{lem}
\begin{proof}
Let $C=A-B$ with
\[
C=\begin{pmatrix}c_{11}&c_{12}\\c_{12}&c_{22}\end{pmatrix}.
\]
If (at least) one diagonal entry of $C$ is zero the statement is true by taking for instance
\[
a=\begin{cases}(0,1)& c_{11}=0,\\(c_{11},2c_{12})& c_{11}\neq0,\end{cases}
\quad \text{and} \quad
b=\begin{cases}(2c_{12},c_{22})& c_{11}=0,\\(1,0)& c_{11}\neq0,\end{cases}
\]
with $a,b\in\R^2$ as in \eqref{eq:comp}.
Assume now that $c_{11},c_{22}\neq0$.
The compatibility condition \eqref{eq:comp} consists in the system
\[
\begin{cases}
c_{11}=a_1b_1,\\
c_{22}=a_2b_2,\\
2c_{12}=a_1b_2+a_2b_1,
\end{cases}
\]
that admits solutions if and only if $\det(C)=c_{11}c_{22}- c_{12}^2\leq0$.
Indeed, suppose for the moment that a solution exists. Then by homogeneity and since $c_{11}\neq 0$, we may without loss
  generality assume that $a_1=1$. By the first equation, it immediately
  follows that $b_1=c_{11}$. Thus, the system reduces to an algebraic problem
  for $a_2,b_2$:
  \begin{align*}
    c_{22}&=a_2b_2,\\
2c_{12}&=b_2+a_2c_{11}.
  \end{align*}
  Further taking differences we obtain that
  \begin{align*}
    c_{22}- 2c_{12}a_{2}&= a_2^2 c_{11} \Leftrightarrow a_2  = \frac{c_{12}\pm \sqrt{c_{12}^2-c_{11}c_{22}}}{c_{11}}, \\
    c_{11}c_{22}- 2 c_{12} b_2 &= -b_2^2\Leftrightarrow b_2 =c_{12} \mp \sqrt{c_{12}^2-c_{11}c_{22}},
  \end{align*}
  for some choice of signs $\pm$, $\mp$.
  These candidates are real valued if and only if
  $c_{11}c_{22}-c_{12}^2\leq 0$.
  Conversely, defining $a_2, b_2$ by the above formulas (with the above matching
  choice of $\mp$ and
  $\pm$) and $a_{1}=1, b_{1}=c_{11}$, we immediately verify that this is indeed
  a solution of the system.
\end{proof}

We next recall the usual identification between $\Stw$ and $\R^3$ via the map
\[
\Stw\ni\begin{pmatrix}a&c\\c&b\end{pmatrix}\mapsto\Big(\frac{a-b}{2},c,\frac{a+b}{2}\Big)\in\R^3.
\]
Equivalently, every $F\in\Stw$ can be parametrized as
\[
F=\begin{pmatrix}z+x&y\\y&z-x\end{pmatrix}, \quad \text{for a unique }(x,y,z)\in\R^3,
\]
see, e.g., \cite{P10, K1}.
Notice that lines in $\Stw$ correspond to lines in $\R^3$.
Note also that with the above parametrization we have $\det(F)=z^2-x^2-y^2$ and $\tr(F)=2z$.

\begin{proof}[Proof of Proposition \ref{prop:2d}]
We show that no triple $A_1,A_2,A_3\in\Stw$ of pairwise incompatible matrices can form a $T_3$-structure  as in Definition \ref{def:T3}.
Lemma \ref{lem:2x2comp} yields that, given a matrix $A\in\Stw$ of coordinates $(x_0,y_0,z_0)$, the set of symmetric matrices that are compatible with $A$ consists of the following cone
\[
\mathcal{C}(A):=\Big\{B=\begin{pmatrix}z+x&y\\y&z-x\end{pmatrix} \,:\, (z-z_0)^2\le(x-x_0)^2+(y-y_0)^2\Big\}.
\]
Hence, $A_2,A_3\in\Stw\setminus\mathcal{C}(A_1)$.
The cone $\Stw\setminus\mathcal{C}(A_1)$ consists of two connected components
\begin{align*}
E_+(A_1)&:=(\Stw\setminus\mathcal{C}(A_1))\cap\{B\in\Stw \,:\, \tr(B-A_1)>0\},\\
E_-(A_1)&:=(\Stw\setminus\mathcal{C}(A_1))\cap\{B\in\Stw \,:\, \tr(B-A_1)<0\}.
\end{align*}
If $A_2$ and $A_3$ both belong to $E_+(A_1)$ ($E_-(A_1)$ resp.), then
by convexity of $E_{\pm}(A_1)$ (and recalling that the parametrization used sends lines to lines) any $\tilde J_1$ as in point (ii) of Definition \ref{def:T3} belongs to $E_+(A_1)$ ($E_-(A_1)$ resp.) as well.
So $\tilde J_1$ is incompatible with $A_1$ and the triple $A_1,A_2,A_3$ does not form a $T_3$-structure.
Therefore, without loss of generality, we can assume that $A_2\in E_+(A_1)$ and $A_3\in E_-(A_1)$.
Analogously, we also have $A_1,A_3\in\Stw\setminus\mathcal{C}(A_2)$ but $\tr(A_1-A_2)\le0$ and $\tr(A_3-A_2)=\tr(A_3-A_1)+\tr(A_1-A_2)\le0$.
This implies that $A_1,A_3\in E_-(A_2)$ and thus there exists no $\tilde{J}_2$ as in point (ii) of Definition \ref{def:T3}.
\end{proof}

\section{On rigidity of the example from \texorpdfstring{\cite{BFJK94}}{BFJK94a}}
\label{sec:BFJK}

In this section, we turn to the main qualitative result, i.e., the proof of Theorem
\ref{thm:rigidity0}. The argument will also play an important role
as a starting point for the quantified version of this which we discuss in Part \ref{part:2} below.

\subsection{The structure of three-dimensional diagonal strains}
\label{sec:struc}

Seeking to study the differential inclusion \eqref{eq:BFJK}, in the following we consider $d=3$ and let $\Omega\subset\R^3$ be a simply
  connected Lipschitz domain.
Following the arguments from \cite{DM2}, we first infer some structure on deformations with constant divergence and consisting of diagonal strains.
More precisely, \cite[Lemma 3.2]{DM2} provides a decomposition of diagonal strains with constant trace into six planar waves of the following form:

\begin{prop}[Lemma 3.2 in \cite{DM2}]\label{prop:structure-def}
Let $\Omega \subset \R^3$ be an open, simply connected domain.
Let $c\in\R$ and let $K(c)$ be a subset of $\{M\in \R^{3\times 3}_{diag}: \ \tr(M)=c\}$.
Let $u\in W^{1,\infty}(\Omega;\R^3)$ satisfy
\begin{equation}\label{eq:inclusion-3D}
e(u)\in K(c) \mbox{ a.e. in } \Omega.
\end{equation}
Then there exist functions $F_{ij}, G_{ij}:\R\to\R$ for $i,j\in\{1,2,3\}$,
$i\neq j$ and a matrix $M\in K(c)$, $M=\diag(m_1,m_2,m_3)$ such that there holds
\begin{equation}\label{eq:structure}
\begin{split}
u_1(x_1,x_2,x_3) &= F_{12}(x_1+x_2)+G_{12}(x_1-x_2)+F_{13}(x_1+x_3)+G_{13}(x_1-x_3)+m_1 x_1, \\
u_2(x_1,x_2,x_3) &= F_{21}(x_1+x_2)+G_{21}(x_1-x_2)+F_{23}(x_2+x_3)+G_{23}(x_2-x_3)+m_2 x_2, \\
u_3(x_1,x_2,x_3) &= F_{31}(x_1+x_3)+G_{31}(x_1-x_3)+F_{32}(x_2+x_3)+G_{32}(x_2-x_3)+m_3 x_3,
\end{split}
\end{equation}
and the functions $F_{ij}, G_{ij}$ comply with 
\begin{equation}\label{eq:Fcomp}
\begin{split}
F'_{ij}(z) &= -F'_{ji}(z), \quad \text{for every } z\in\{x_i+x_j:\  x\in\Omega\}, \\
G'_{ij}(z) &= G'_{ji}(z), \quad \text{for every } z\in\{x_i-x_j: \ x\in\Omega\}.
\end{split}
\end{equation}
\end{prop}

\begin{rmk}[Structure of diagonal strains]
Under the hypotheses of Proposition \ref{prop:structure-def}, it is
straightforward to obtain the following structure at the level of the strains:
there exist six functions $f_{ij},g_{ij}:\R\to\R$ for $i,j\in\{1,2,3\}$,
$i<j$, and a matrix $M\in K(c)$, $M=\diag(m_1,m_2,m_3)$ such that there holds
\begin{equation}\label{eq:strain-structure}
\begin{split}
e_{11}(x_1,x_2,x_3) &= f_{12}(x_1+x_2)+g_{12}(x_1-x_2)-f_{13}(x_1+x_3)-g_{13}(x_1-x_3)+m_1,\\
e_{22}(x_1,x_2,x_3) &= -f_{12}(x_1+x_2)-g_{12}(x_1-x_2)+f_{23}(x_2+x_3)+g_{23}(x_2-x_3)+m_2,\\
e_{33}(x_1,x_2,x_3) &= f_{13}(x_1+x_3)+g_{13}(x_1-x_3)-f_{23}(x_2+x_3)-g_{23}(x_2-x_3)+m_3.
\end{split}
\end{equation}
\end{rmk}

For self-containedness, we present a short sketch of the proof of the decomposition on the level of the strain.

\begin{proof}[Proof of the decomposition from \eqref{eq:strain-structure}]
Let $M \in K(c)$, then by considering $u(x)-Mx$ and $K(0):=K(c)-M$ we may
  without loss of generality assume that $c=0$.
By linearity we can assume $u$ to be smooth.
Now the compatibility conditions for diagonal strains read
\begin{align}
\label{eq:comp1}
\begin{split}
\p_{23} e_{11} & = 0,\\
\p_{13} e_{22} & = 0,\\
\p_{12} e_{33} & = 0,
\end{split}
\end{align}
and 
\begin{align}
\label{eq:comp2}
\begin{split}
\p_{22} e_{11} + \p_{11} e_{22} & = 0,\\
\p_{33} e_{22} + \p_{22} e_{33} & = 0,\\
\p_{11} e_{33} + \p_{33} e_{11} & = 0.
\end{split}
\end{align}
Using first the equations in \eqref{eq:comp1}, we obtain that
\begin{align*}
e_{11}(x_1,x_2,x_3) & = b_{12}(x_1,x_2) + b_{13}(x_1,x_3),\\
e_{22}(x_1,x_2,x_3) & = b_{21}(x_1,x_2) + b_{23}(x_2,x_3),\\
e_{33}(x_1,x_2,x_3) & = b_{31}(x_1,x_3) + b_{32}(x_2,x_3).
\end{align*}
Using these dependences, the constant trace condition implies that $\p_{12} b_{12} + \p_{12} b_{21} = 0$.
Inserting this into the first equation in \eqref{eq:comp2} differentiated with respect to $x_1, x_2$, we obtain that
\begin{align*}
(\p_{11}-\p_{22})\p_{12} b_{12} = 0.
\end{align*}
Hence,
\begin{align*}
b_{12}(x_1, x_2) = f_{12,+}(x_1+x_2) + f_{12,-}(x_1-x_2) + b_{121}(x_1) + b_{122}(x_2),
\end{align*}
and 
\begin{align*}
b_{21}(x_1, x_2) = -f_{12,+}(x_1+x_2) - f_{12,-}(x_1-x_2) + b_{211}(x_1) + b_{212}(x_2),
\end{align*}
with some one-dimensional functions $b_{121}, b_{122}, b_{211}, b_{212}$.
Analogous results hold for the other functions in the decomposition above. 

It thus remains to argue that the functions of the single variables $x_1, x_2$ are constant. For $b_{122}(x_2), b_{211}(x_1)$ some further information follows from the first equation in \eqref{eq:comp2}. Indeed, it follows that
\begin{align}
\label{eq:quadratic}
b_{122}''(x_2) + b_{211}''(x_1) = 0,
\end{align}
which implies that both functions are second order polynomials (with leading order coefficient $a \in \R$ and $-a$, respectively). As a consequence, arguing similarly in all other variables, we obtain the decomposition
\begin{equation}\label{eq:struct1}
\begin{split}
e_{11}(x_1,x_2,x_3) & = f_{12,+}(x_1+x_2) + f_{12,-}(x_1-x_2) +f_{13,+}(x_1+x_3) + f_{13,-}(x_1-x_3)\\ & \quad +P_{12}(x_2)+ P_{13}(x_3)+ f_1(x_1),\\
e_{22}(x_1,x_2,x_3) & = -f_{12,+}(x_1+x_2) - f_{12,-}(x_1-x_2) +f_{23,+}(x_2+x_3) + f_{23,-}(x_2-x_3)\\ & \quad +P_{21}(x_1)+ P_{23}(x_3) + f_2(x_2),\\
e_{33}(x_1,x_2,x_3) & =- f_{13,+}(x_1+x_3) - f_{13,-}(x_1-x_3) -f_{23,+}(x_2+x_3) - f_{23,-}(x_2-x_3)\\& \quad +P_{31}(x_1)+ P_{32}(x_2)+ f_3(x_3),
\end{split}
\end{equation}
where $P_{ij}$ are second order polynomials complying with
\begin{equation}\label{eq:poly-coeff}
P_{ij}''=-P_{ji}'', \quad i,j\in\{1,2,3\},\, i\neq j.
\end{equation}
It remains to infer information on the functions $f_j(x_j)$ and to reduce the order of the polynomials to be constants. This again follows from the trace constraint which becomes
\begin{align*}
0 &= e_{11} + e_{22} + e_{33} \\
& = P_{12}(x_2)+ P_{13}(x_3)+P_{21}(x_1)+ P_{23}(x_3)+ P_{31}(x_1)+ P_{32}(x_2)+ f_1(x_1) + f_2(x_2) + f_3(x_3).
\end{align*}
Varying the individual variables, for every $j\in\{1,2,3\}$ we obtain that
\begin{equation}\label{eq:extra-term}
f_j(x_j) = -P_{ij}(x_j) - P_{kj}(x_j) + c_j,
\end{equation}
for some $c_j \in \R$ with $c_1+c_2+c_3=0$.
With this in hand, we claim that there exist functions $f_{ij},g_{ij}$ satisfying the structure equations in \eqref{eq:strain-structure}.
Indeed, it remains to argue that the quadratic polynomials in the structure of the strains $e_{jj}$ can be rewritten as functions in the variables $x_j\pm x_i$ and $x_j\pm x_k$.
For the sake of clarity of exposition we show this for $j=1$, while for $j=2,3$ the computations are completely analogous.
From \eqref{eq:struct1} and \eqref{eq:extra-term} we have
\begin{align*}
e_{11}(x_1,x_2,x_3) &= f_{12,+}(x_1+x_2) + f_{12,-}(x_1-x_2) +f_{13,+}(x_1+x_3) + f_{13,-}(x_1-x_3) \\
& \quad +P_{12}(x_2)+ P_{13}(x_3)-P_{21}(x_1)-P_{31}(x_1)+c_1.
\end{align*}
We first note that all linear contributions can be rewritten as $x_1 = \frac{1}{2}(x_1 + x_i) + \frac{1}{2}(x_1 - x_i)$, $i\in\{2,3\}$ and thus can be absorbed in the functions $f_{1i,\pm}$, obtaining
\begin{align*}
e_{11}(x_1,x_2,x_3) &= \tilde f_{12,+}(x_1+x_2) + \tilde f_{12,-}(x_1-x_2) +\tilde f_{13,+}(x_1+x_3) + \tilde f_{13,-}(x_1-x_3) \\
& \quad +P_{12}''x_2^2+ P_{13}''x_3^2-P_{21}''x_1^2-P_{31}''x_1^2,
\end{align*}
for some $\tilde f_{1i,\pm}$.
Eventually, from \eqref{eq:poly-coeff} we have that
$P_{1i}''=-P_{i1}''$, hence the quadratic contributions can be rewritten in the desired structure since $x_1^2 + x_i^2 = \frac{1}{2}(x_1+ x_i)^2 + \frac{1}{2}(x_1-x_i)^2$.
Defining $f_{1i}, g_{1i}$ such that these contributions are included then implies \eqref{eq:strain-structure}.
\end{proof}

\subsection{Proof of Theorem \texorpdfstring{\ref{thm:rigidity}}{3} -- Rigidity for the \texorpdfstring{\cite{BFJK94}}{BJFK94a} example}

We next continue by exploiting the wave structure from above, the three-valuedness of the strains and the fact that the strain components determine each other.

Before turning to the proof of Theorem \ref{thm:rigidity}, we discuss some auxiliary results.

\begin{lem}
\label{lem:determined}
Let $\Omega \subset \R^3$ be an open domain.
Let $e_{11}, e_{22}, e_{33}: \Omega \rightarrow \R$ be a solution of \eqref{eq:BFJK} and let \eqref{eq:elliptic} hold. Then, for any $x,y\in \Omega$ and $j,k\in\{1,2,3\}$, the difference
\begin{align*}
e_{jj}(x)-e_{jj}(y) \mbox{ uniquely determines } e_{kk}(x)-e_{kk}(y).
\end{align*}
\end{lem}

We view this lemma as encoding the determinedness structure of the wells from \eqref{eq:BFJK}. In particular, it is a highly nonlinear ingredient in our argument.

\begin{proof}
We note that $e_{jj}(x)-e_{jj}(y)$ can only attain the values
\begin{align}
  \label{eq:values1}
0, \pm (\eta_3 - \eta_2), \pm (\eta_1 - \eta_2), \pm (\eta_3 - \eta_1).
\end{align}
If the difference is zero, then by determinedness, this is also the case for the $e_{kk}$ difference. Without loss of generality, we may thus assume that it is not zero. 

We claim that in the remaining cases, the  values listed in
  \eqref{eq:values1} are all distinct and that hence the difference
$e_{jj}(x)-e_{jj}(y) $ already determines $e_{jj}(x)$ and $e_{jj}(y)$ which then
by determinedness also determines $e_{kk}(x)$ and $e_{kk}(y)$ (and thus in particular the difference).
Indeed, by the assumption \eqref{eq:elliptic} we have that
\begin{align*}
\eta_3 - \eta_2 > \eta_1-\eta_2>0, \ \eta_3-\eta_2 > \eta_3-\eta_1 >0.
\end{align*}
We further claim that 
\begin{align*}
 \eta_3-\eta_1 >  \eta_1-\eta_2.
\end{align*}
Indeed, this is equivalent to the second assumption in \eqref{eq:elliptic}. As a consequence,
\begin{align}
\label{eq:values}
\eta_3-\eta_2 > \eta_3-\eta_1 > \eta_1 - \eta_2>0,
\end{align}
and thus all values in \eqref{eq:values1} are distinct.
After possibly exchanging $x,y$, we may assume that $e_{jj}(x)-e_{jj}(y)>0$ and that this difference attains one of the values from \eqref{eq:values}. This however uniquely determines $e_{jj}(x), e_{jj}(y)$ which concludes the proof.
\end{proof}

\begin{lem}
\label{lem:two_waves}
Let $b: \R^2 \rightarrow \R$, $b=b(x_2,x_3)$, be such that
\begin{align*}
\p_2 \p_3 b = 0, \ (\p_{22} - \p_{33}) b = 0 \mbox{ in } \mathcal{D}'.
\end{align*}
Then, $b$ must be a quadratic function in $x_2, x_3$.
\end{lem}

\begin{proof}
It suffices to show that $D^3 b = 0$.
To this end, we note that by differentiating the first equation for $b$, we
obtain that all mixed third order derivatives vanish. Combining this with
derivatives of the second equation also implies that the remaining two
third order derivatives vanish.
\end{proof}

\begin{proof}[Proof of Theorem \ref{thm:rigidity}]
\emph{Step 1.}
We first claim that for any $h_2, h_3 \in \R$ sufficiently small (so that $B_{|h_2|+|h_3|}(x)\subset \Omega$) the iterated differences satisfy
\begin{align}
\label{eq:step1}
\partial_2^{h_2} \partial_3^{h_3} e_{22} (x) = 0, 
\end{align}
where $\partial_j^{h_j} f(x) = f(x + h_j e_j)- f(x)$ and where $\{e_1,e_2,e_3\}$ denotes the canonical unit basis of $\R^3$.

Indeed, by the wave structure from \eqref{eq:strain-structure}, we note that 
\begin{align*}
\p_2^{h_2} \p_3^{h_3} e_{11} = 0.
\end{align*}
By this equation for $e_{11}$ we have $ \p_3^{h_3} e_{11}(x + h_2 e_2) =  \p_3^{h_3} e_{11}(x )$. Invoking Lemma \ref{lem:determined} we then also have that $ \p_3^{h_3} e_{22}(x + h_2 e_2) =  \p_3^{h_3} e_{22}(x )$ which shows the claim \eqref{eq:step1}.\\

\emph{Step 2.} 
Using the weak formulation, we conclude from step 1 and the wave structure of $e_{22}$ that
\begin{align*}
\p_{2}\p_3 b_{23} = 0, \ (\p_{22} - \p_{33}) b_{23} =0,
\end{align*}
where $b_{23}(x_2,x_3) := f_{23}(x_2 + x_3) + g_{23}(x_2-x_3)$.
By Lemma \ref{lem:two_waves} this however implies that $b_{23}$ is a quadratic function in $x_2, x_3$.
We claim that since $e_{22}$ and $ e_{33}$ are discrete, we obtain that $b_{23}$ must already be constant. Indeed, fixing $x_1, x_2$ we note that $b_{23}(x_2, \cdot) = e_{22}(x_1,x_2, \cdot)+f_{12}(x_1+x_2) + g_{12}(x_1-x_2)$ attains at most three values as a function of $x_3$. A similar argument using the discreteness of $e_{33}$ shows that $b_{23}(\cdot, x_3)$ is discrete. Being both quadratic and discrete in each variable, $b_{23}$ must already be constant.

The proof is now concluded by noting that the same argument applies to all other plane waves in the decomposition \eqref{eq:strain-structure}.
\end{proof}

\subsection{An alternative decomposition}
\label{sec:notes}

In this section we provide an alternative proof of Theorem \ref{thm:rigidity},
where instead of the decomposition into plane waves as in Proposition
\ref{prop:structure-def}, we establish a decomposition into functions of single variables.

\begin{prop}
  \label{prop:xdecomp}
  Let $\Omega \subset \R^3$ be an open, simply connected domain and let $e=e(u)$ be a
  given solution of
  \begin{align*}
    e(u) \in K_3  \mbox{ a.e.~in } \Omega.
  \end{align*}
  Then there exist nine functions $f_{i}^{j}(x_j)$, $i,j=1,2,3$, which depend only on
  one variable, such that for $i=1,2,3$ we have the decomposition
  \begin{align}
    \label{eq:coordinatewaves}
    e_{ii}= \sum_{j=1}^3 f_{i}^j(x_j).
  \end{align}
\end{prop}
We remark that while these single variable functions are conceptually simpler
than the plane waves of Proposition \ref{prop:structure-def}, there is a larger
number of them. Furthermore, this result requires more structure of the set $K_3$
in order to make use of Lemma \ref{lem:determined} (in that already in deducing the decomposition the determinedness properties are invoked).
\begin{proof}
We argue similarly as in the proof of Lemma \ref{lem:two_waves} and make use of
the compatibility conditions at the level of difference quotients.
We thus recall that part of the compatibility conditions read
\begin{align*}
  \p_{12} e_{33}=0,\\
  \p_{13} e_{22}=0,\\
  \p_{23} e_{11}=0.
\end{align*}
In particular, for $x\in \Omega$ and any $h_1,h_2 \neq 0$ with $B_{|h_1|+|h_2|}(x)\subset \Omega$ the first condition implies that the
corresponding iterated finite differences of $e_{33}$ vanish:
\begin{align*}
  \p^{h_1}_{1} \p^{h_2}_{2} e_{33}(x) &=0 \\
  \Leftrightarrow  \p^{h_2}_{2}e_{33} (x_1,x_2,x_3) &= \p^{h_2}_{2}e_{33} (x_1+h_1,x_2,x_3).
\end{align*}
By the results of Lemma \ref{lem:determined} both differences also uniquely
determine $\p^{h_2}_{2}e_{22}$ at the respective point and thus also the
finite differences of $e_{22}$ agree:
\begin{align*}
  \p^{h_2}_{2} e_{22}(x_1,x_2,x_3) &= \p^{h_2}_{2} e_{22}(x_1+h_1,x_2,x_3) \quad \forall h_1, h_2\\
  \Rightarrow  \p_{12} e_{22}(x) &=0.
\end{align*}
Repeating this argument for $\p_{23} e_{11}$, it follows
that
\begin{align*}
  \p_{12}e_{22}=\p_{13}e_{22} = \p_{23}e_{22}=0,
\end{align*}
and hence $e_{22}$ has the desired form. The argument for $e_{11}$
and $e_{33}$ is analogous.
\end{proof}

Given this wave decomposition and using the remaining compatibility conditions
we deduce rigidity.
\begin{prop}
  Let $e$ be as in Proposition \ref{prop:xdecomp}, then $e_{ii}$ is constant.
\end{prop}

\begin{proof}
  We recall the second part of the compatibility conditions
  \begin{align*}
    \p_{11}e_{22}+ \p_{22}e_{11}&=0, \\
    \p_{11}e_{33} + \p_{33}e_{11}&=0, \\
    \p_{22}e_{33} + \p_{33}e_{22}&=0.
  \end{align*}
  Inserting the decomposition \eqref{eq:coordinatewaves} of Proposition
  \ref{prop:xdecomp} it follows that
  \begin{align*}
    \p_{jj} f^{j}_{i} (x_j) + \p_{ii} f^{i}_{j}(x_i)=0
  \end{align*}
  for any $i\neq j$. Since both functions depend on different variables, both
  functions are necessarily constant and hence $f_{i}^j$ is a quadratic
  function.
  That is, the decomposition has the form
  \begin{align*}
    e_{11}(x) &= f_1^1(x_1) + q_2^1 (x_2) + q_3^1(x_3), \\
    e_{22}(x) &= f_2^2(x_2) + q_1^2 (x_1) + q_3^2(x_3), \\
    e_{33}(x) &= f_3^3(x_3) + q_2^3 (x_2) + q_1^1(x_1),
  \end{align*}
  where the functions $q_{i}^j$ are quadratic functions and $f_{i}^{i}(x_i)$ is
  arbitrary.
  
  Furthermore, since $e_{11},e_{22},e_{33}$ are discrete, by restricting to
  suitable hypersurfaces, it follows that all quadratic functions are discrete and hence constant.
  Therefore all $e_{ii}=e_{ii}(x_i)$ are functions of only one distinct variable.
  Finally, since $e_{ii}$ determines $e_{jj}$ this can only be
  possible if each $e_{ii}$ is constant.
\end{proof}

\part{Quantitative Results}
\label{part:2}

In this section, we turn to the proofs of the quantitative results for the $T_3$-structures from above.
In this context, compared to the setting of the Tartar square which was studied in \cite{RT22}, we make use of three major novel steps. These become necessary due to the more complicated well-structure and wave decomposition in the presence of gauge invariances.
\begin{itemize}
\item Firstly, we need to isolate the most relevant planar dependences for the respective strain components. This is achieved by a finite difference argument which (up to controlled errors) allows us to ignore the remaining plane waves. In this context, it is important that the step size $h>0$ in the finite difference operator is chosen to be sufficiently small.
\item Secondly, we use the determinedness to relate one of the other strain components to the original strain component in order to reduce the Fourier support of the functions forming the decomposition of the strain. At this point, we rely on several nonlinear estimates combined with the a priori bounds for the ``irrelevant'' waves.
\item While the first two steps are mainly carried out on certain finite differences of the strain components, in the third step, we finally return to the original waves from the finite difference control by using the localization of the individual waves. Here we require a lower bound on the finite difference step size $h>0$.
\end{itemize}
The combination of isolating the ``relevant waves'' (and thus partially the breaking of symmetry), a careful choice of the finite difference step-size parameter which is balanced exactly in such a way that both the upper and lower bounds are guaranteed, and the ``ellipticity'' arguments which had been introduced in the context of the Tartar square, eventually lead to the desired lower bounds.

We emphasize that the choice of the finite difference step-size $h$ will depend on each stage of
the iteration process, i.e., at the $m$-th iteration we choose $h \sim \frac{1}{\mu^m \mu_1}$ with $\mu, \mu_1$ being constants which are introduced and determined in what follows below.

\section{Preliminary Results}
\label{sec:pre2}

In this section, we recall properties of our $T_3$-structure which we will rely on in the following quantitative considerations.

\subsection{The auxiliary matrices, (symmetrized) rank-one directions} 
We first recall the structure of the matrices $e^{(1)},e^{(2)}, e^{(3)}$. 
To this end, we introduce auxiliary symmetrized rank-one connections although the set of stress-free strains itself is elliptic.

\begin{lem}[Lemma 3.3 in \cite{BFJK94}]
\label{lem:BFJK94}
Let $\eta_1, \eta_2,\eta_3$ satisfy \eqref{eq:elliptic}.
Let $J_1 = \diag(\kappa, \eta_1, \eta_1)$, $J_2 = \diag(\eta_1, \kappa, \eta_1)$, $J_3 = \diag( \eta_1, \eta_1, \kappa)$ with $\kappa = \eta_2 + \eta_3-\eta_1$. Then the matrices $J_i, J_{\ell}$ are pairwise symmetrized rank-one connected with
\begin{align*}
  J_{\ell}-J_{i} = (\kappa-\eta_1) \epsilon_{i \ell m } \left( b_{i \ell} \otimes b_{\ell i} + b_{\ell i} \otimes b_{i \ell} \right),
\end{align*}
where $i,\ell,m\in\{1,2,3\}$, 
and where
\begin{align}
\label{eq:sym_rank_one}
\begin{split}
&b_{12} = \frac{1}{\sqrt{2}} \begin{pmatrix}
1 \\ 1 \\0
\end{pmatrix}, \
b_{31} = \frac{1}{\sqrt{2}} \begin{pmatrix}
1 \\ 0 \\1
\end{pmatrix}, \
b_{23} = \frac{1}{\sqrt{2}} \begin{pmatrix}
0 \\ 1 \\1
\end{pmatrix},\\
&b_{21} = \frac{1}{\sqrt{2}} \begin{pmatrix}
- 1 \\ 1 \\ 0
\end{pmatrix}, \
b_{13} = \frac{1}{\sqrt{2}} \begin{pmatrix}
1 \\ 0 \\ -1
\end{pmatrix},\
b_{32} = \frac{1}{\sqrt{2}} \begin{pmatrix}
0 \\ -1 \\1
\end{pmatrix},
\end{split}
\end{align}
and 
\begin{align*}
\epsilon_{i\ell m} = 
\left\{
\begin{array}{ll}
1 \mbox{ if } (i, \ell, m) \mbox{ is an even permutation of } (1,2,3),\\
-1 \mbox{ if } (i, \ell, m) \mbox{ is an odd permutation of } (1,2,3),\\
0 \mbox{ else}.
\end{array} \right.
\end{align*}
Moreover, for $i \in\{1,2,3\}$ 
\begin{align*}
&e^{(i)}-J_{i} = (\eta_1-\eta_2)\left( b_{i j} \otimes b_{j i} + b_{j i} \otimes b_{i j} \right),\\
&J_{j}- e^{(i)} = (\eta_{1}-\eta_{3})\left( b_{i j} \otimes b_{j i} + b_{j i} \otimes b_{i j} \right),
\end{align*}
where we set $j:=(i-1)$ {\rm mod} $(3)$.
\end{lem}

We refer to Figure \ref{fig:T3} for an illustration of this.

For convenience of notation, we collect the symmetrized rank-one directions into the following sets
\begin{align}
\label{eq:normals}
\mathcal{B}_{i \ell}:= \left\{ b_{i \ell}, b_{\ell i} \right\}, \ \mathcal{B}_{\ell}:= \mathcal{B}_{i \ell}\cup \mathcal{B}_{\ell j}= \left\{b_{i \ell}, b_{\ell i}, b_{\ell j}, b_{j \ell} \right\}, \ \mathcal{B}:=\{b_{12},b_{21},b_{13},b_{31},b_{23},b_{32}\}.
\end{align}

To shorten notation, we also write
\begin{equation}\label{eq:data-sets}
\mathcal{E}:=\conv\{J_1,J_2,J_3\}\cup \bigcup\limits_{j=1}^3 [J_j, e^{(j)}].
\end{equation}
It is known \cite[Proposition 4.9]{CS} that $\mathcal{E}$ is contained in the \emph{symmetrized rank-one convex hull} of $K_3$.

\subsection{Notation and preliminary results in Fourier analysis}
In the following, as done in, e.g.,\ \cite{KKO13,CO,RT22}, we will work in Fourier space making use of \emph{frequency localization techniques} and \emph{multiplier theorems}.
We therefore introduce some notation and known results that will be useful in our analysis.

We use the following notation for the Fourier transform: Given $u\in L^1(\T^3)$,
$$
\F u(k) := \int_{\T^3} e^{-2\pi i k\cdot x}u(x) dx, \quad k\in\Z^3
$$
denotes the $k$\emph{-th Fourier coefficient} of $u$.
If there is no ambiguity we will use the notation $\hat u=\F u$.

A bounded function $m:\Z^3\to\R$ gives rise to a \emph{Fourier multiplier} $m(D)$ defined as
$$
\F(m(D)u)(k):=m(k)\hat u(k), \quad \text{for every } k\in\Z^3,
$$
for $u\in L^1(\T^3)$.
We now define a class of functions whose restriction to $\Z^3$ is associated to a Fourier multiplier (thanks to the transference principle, see, e.g.,\ \cite[Theorem 4.3.7]{Grafakos}) which is $L^p$-$L^p$ bounded for any $p\in(1,\infty)$.

Given a constant $A>0$, we define $\mathcal{M}_A\subset C^\infty_c(\R^3 \setminus \{0\})$ as follows: $m\in\mathcal{M}_A$ if and only if 
there exists an orthonormal basis $\{v_1,v_2,v_3\}$ such that for every $\alpha\in\N^3$ with $|\alpha|\le3$
$$
|\p_{v_1}^{\alpha_1}\p_{v_2}^{\alpha_2}\p_{v_3}^{\alpha_3} m(\xi)|\le A|\xi\cdot v_1|^{-\alpha_1}|\xi\cdot v_2|^{-\alpha_2}|\xi\cdot v_3|^{-\alpha_3},
$$
for every $\xi\in\R^3$ such that $\xi\cdot v_j\neq0$, $j\in\{1,2,3\}$.

By combining the transference principle \cite[Theorem 4.3.7]{Grafakos} and a corollary of the Marcinkiewicz's multiplier Theorem \cite[Corollary 6.2.5]{Grafakos}, we obtain the following result.

\begin{lem}
\label{lem:mult-thm}
Let $A>0$ and let $\mathcal{M}_A$ be defined as above.
Then, for any $p\in(1,\infty)$, for every $m\in\mathcal{M}_A$ there holds
\begin{align}
\label{eq:marcinkiewicz-gen}
\|m(D) f\|_{L^p} \leq C A \max\left\{ p , \frac{1}{p-1} \right\}^{18} \|f\|_{L^p},
\end{align}
for every $f\in L^p(\T^3;\C)$.
\end{lem}

In the following, for every $k\in\R^3$ we will denote $\hat k:=\frac{k}{|k|}$ for $k\neq 0$ and $0$ otherwise.

\subsection{Elastic energy decomposition, high frequency bounds and nonlinear reduction argument} 
In this section, we begin to translate the qualitative arguments from Part \ref{part:1} into quantitative results. To this end, we begin by estimating the elastic energy from below. For every $\bar{e}\in\R^{3\times3}_{sym}$, setting $v(x) = u(x)-\bar{e} x$ and $\tilde{\chi}=\chi- \bar{e}$, we obtain
\begin{align}
\label{eq:el_chi}
\begin{split}
E_{el}(\chi)
&:= \inf\limits_{u \in \mathcal{A}_{\bar{e}}} \int\limits_{\T^3} |e( u) - \chi|^2 dx
= \inf\limits_{v \in \mathcal{A}_0(\T^3, \R^3)} \int\limits_{\T^3} |e( v)-\tilde{\chi}|^2 dx\\
&=\inf\limits_{v \in H^1(\T^3, \R^3)}\int\limits_{\T^3} |e( v)- \tilde{\chi}|^2 dx.
\end{split}
\end{align}

Building upon the above notation, we also introduce the \emph{minimal total energy} for a given $\chi \in BV(\T^3, K_3)$ as
\begin{equation}\label{eq:e-eps}
E_\epsilon(\chi) := E_{el}(\chi)+\epsilon E_{surf}(\chi),
\end{equation}
where $E_{surf}$ is defined as in \eqref{eq:surf}.

As in \cite{KKO13, CO}, we next consider the Fourier multiplier associated with the elastic energy $E_{el}(\chi)$.
This will subsequently allow us to decompose the diagonal components of the strain into six one-dimensional waves (in parallel to Proposition \ref{prop:xdecomp} in the stress-free setting).

\begin{lem}[Elastic energy control, Lemma 4.2 in \cite{KKO13} and Lemma 3.1 in \cite{CO1}]
\label{lem:elast2}
Let $\chi\in L^\infty(\T^3;K_3)$, $\bar{e}\in\R^{3\times3}_{sym}$ and $\tilde\chi=\chi-\bar{e}$.
Let $E_{el}$ be as in \eqref{eq:el_chi} and
let $\mathcal{B}_j$ denote the sets from \eqref{eq:normals}. Then,
\begin{align*}
E_{el}(\chi) \gtrsim \sum\limits_{j=1}^3 \sum\limits_{k \in \Z^3} \dist^2\big(\hat k, \mathcal{B}_j\big) |\F\tilde{\chi}_{jj}(k)|^2.
\end{align*}
In particular $|\langle\chi\rangle-\bar{e}|^2=|\F\tilde{\chi}(0)|^2\lesssim E_{el}(\chi).$
\end{lem}

\begin{proof}
The proof is identical to the one from \cite[Lemma 3.1]{CO1}. It follows by computing the Euler-Lagrange equations associated with the multiplier and using the diagonal structure of $\chi$.
\end{proof}

Complementing the elastic energy estimates, we also recall the high frequency bounds (c.f., for instance, \cite{CO1,RT22}). 

\begin{lem}[High frequency bounds]
\label{lem:high_freq}
Let $\chi\in BV(\T^3;K_3)$, $\bar{e}\in\R^{3\times3}_{sym}$ and $\tilde\chi=\chi-\bar{e}$.
Let $E_{surf}(\tilde{\chi})$ be as in \eqref{eq:surf} and let $\mu_1> 0$. Then,
\begin{align*}
\sum_{\{|k|\geq \mu_1 \}}|\F\tilde{\chi}(k)|^2  \lesssim \mu_1^{-1} \|\tilde{\chi}\|_{L^{\infty}} E_{surf}(\tilde{\chi}).
\end{align*}
\end{lem}

\begin{proof}
We refer to \cite[Proof of Lemma 4.3, Step 2]{KKO13} or \cite[Lemma 2]{RT22} for a proof of this.
\end{proof}

\begin{figure}[t]
  \centering
  \includegraphics[width=0.5\linewidth]{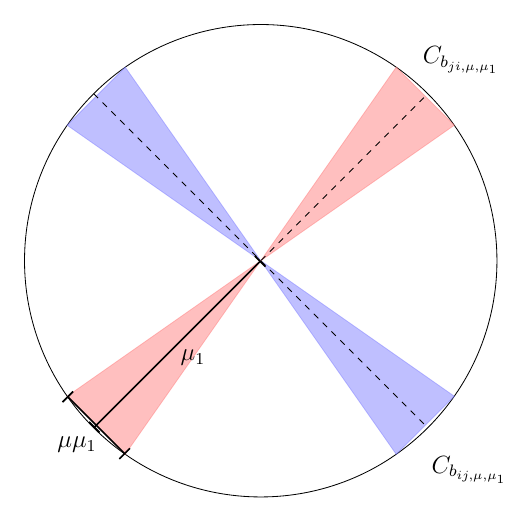}
  \caption{Illustration of the truncated cones defined in \eqref{eq:Cbij}. As $\mu>0$ is
    small, the cones in direction $b_{ij}$ and $b_{ji}$ are disjoint
  except for the origin.}
  \label{fig:conicallocalization}
\end{figure}
Building on Lemma \ref{lem:elast2} and Lemma \ref{lem:high_freq}, we introduce the following conical localization functions: For parameters $0<\mu<1<\tilde{\mu}$ and for $i,j\in\{1,2,3\}$ with $i\neq j$, we define the truncated cones in direction $b_{ij} \in \mathcal{B}$ as
\begin{align}
\label{eq:Cbij}
C_{b_{ij},\mu,\tilde \mu}:=\{k \in \R^3: |\hat k|^2-(b_{ij}\cdot \hat{k})^2 \leq \mu^2, \ |k|\leq \tilde \mu\}
\end{align}
(see Figure \ref{fig:conicallocalization}).
Associated with this we define functions $\chi_{b_{ij},\mu,\tilde \mu}\in
C^{\infty}_c(\R^3\setminus \{0\};[0,1])$ to be non-negative, to be equal to one in $C_{b_{ij},\mu,\tilde \mu}\setminus B_{\frac{1}{2}}$, to vanish outside of $C_{b_{ij},2\mu,2\tilde \mu}\setminus B_{\frac{1}{4}}$, to be symmetric with respect to reflections at the origin. For  $\ell \in \{1,2,3\}\setminus \{i,j\}$ we require it to satisfy the bound
\begin{align}
\label{eq:deriv}
\begin{split}
|\p_{b_{ij}}^{\alpha_1}\p_{b_{ji}}^{\alpha_2}\p_{\ell}^{\alpha_3}\chi_{b_{ij},\mu,\tilde{\mu}}(\xi)|
\leq A|\xi \cdot b_{ij}|^{-\alpha_{1}}|\xi\cdot b_{ji}|^{-\alpha_{2}}|\xi_{\ell}|^{-\alpha_3},\\
\quad \text{for every } \xi\in\R^3 \mbox{ such that } \xi\cdot b_{ij}, \xi\cdot b_{ji}, \xi_\ell\neq0,
\end{split}
\end{align}
for all $\alpha \in \N^3$ with $|\alpha|\le 3$, for some universal constant $A$. 
An example of such multipliers is
$$
\chi_{b_{ij},\mu,\tilde\mu}(k)=\varphi\Big(\frac{\sqrt{|k|^2-|b_{ij}\cdot k|^2}}{\mu|k|}\Big)\varphi\Big(\frac{|k|}{\tilde\mu}\Big)(1-\varphi(4|k|))
$$
where $\varphi\in C^\infty(\R)$ is such that $0\le\varphi\le1$, equals $1$ inside $[-1,1]$ and zero outside $(-2,2)$.
In what follows, with slight abuse of notation, we do not distinguish whether we consider $\chi_{b_{ij},\mu,\tilde \mu}$ as functions on $\R^3$ or on $\Z^3$.

Condition \eqref{eq:deriv} yields that $\chi_{b_{ij},\mu,\tilde\mu}\in\mathcal{M}_A$. 
Thanks to Lemma \ref{lem:mult-thm}, we thus have
\begin{align}
\label{eq:marcinkiewicz}
\|\chi_{b_{ij},\mu,\mu_1}(D) f\|_{L^p}
\leq C A \max\left\{ p , \frac{1}{p-1} \right\}^{18} \|f\|_{L^p},
\end{align}
for every $f\in L^p(\T^3;\C)$.

Building on the bounds for the elastic and surface energy (i.e., Lemmas \ref{lem:elast2} and \ref{lem:high_freq}), we first deduce the following quantitative decomposition of the phase indicator functions. We remark that this is a quantitative version of the wave decomposition result in Proposition \ref{prop:structure-def} in the stress-free setting. A similar, however for our purposes slightly less convenient decomposition can also be found in \cite[Lemma 4.3]{KKO13}.

\begin{lem}[Decomposition]
\label{lem:decomp}
Let $\chi\in BV(\T^3;K_3)$, $\bar{e}\in\mathcal{E}$ and $\tilde\chi=\chi-\bar{e}$.
Let $\mu_0=\mu_0(\mathcal{B})\in(0,1)$, let $0<\mu<\mu_0$,
$\mu_1\geq \frac{1}{2}$ and let $E_\epsilon(\chi)$ be as in \eqref{eq:e-eps}. 
Then the following decomposition holds: There exist six functions $f_{ij},
g_{ij}: \T^3 \rightarrow \R$, $i,j\in\{1,2,3\}$, $i<j$, three functions
$\sigma_{jj}:\T^3 \rightarrow \R$ with $j\in\{1,2,3\}$ and a universal constant $C>0$ such that 
\begin{align*}
\tilde{\chi}_{11} & = f_{12} + g_{12} + f_{13} + g_{13} + \sigma_{11},\\
\tilde{\chi}_{22} & = -f_{12} - g_{12} + f_{23} + g_{23} + \sigma_{22},\\
\tilde{\chi}_{33} & = -f_{13} - g_{13} - f_{23} - g_{23} + \sigma_{33},
\end{align*}
with 
$\supp(\mathcal{F} f_{ij}) \subset C_{b_{ij}, 2\mu, 2\mu_1}\setminus\{0\}$,
$\supp(\mathcal{F} g_{ij}) \subset  C_{b_{ji}, 2\mu, 2\mu_1}\setminus\{0\}$, 
and
\begin{align*}
\sum\limits_{j=1}^{3} \int\limits_{\T^3} |\sigma_{jj}|^2 dx \leq C (\mu^{-2}  + \mu_1^{-1} \epsilon^{-1}) E_{\epsilon}(\chi).
\end{align*}
\end{lem}

We emphasize that the main difference between our plane wave decomposition and the one from \cite{KKO13} consists in the fact that we have a tighter localization than in \cite{KKO13} which is tailored to the cones $C_{b_{ij},\mu,\mu_1}$. For our subsequent iteration argument this will be of central relevance.

\begin{proof}
The claim follows from the estimates for the high frequencies by means of the surface energy (Lemma \ref{lem:high_freq}) together with the observation that
\begin{align*}
E_{el}(\chi) \gtrsim \sum\limits_{j=1}^3 \sum\limits_{k \in \Z^3} \dist^2\big(\hat k, \mathcal{B}_j\big) |\F \tilde\chi_{jj}(k)|^2,
\end{align*}
from Lemma \ref{lem:elast2} and the constant trace condition.

Let us elaborate on this and let us first consider $\tilde{\chi}_{11}$ for simplicity. Combining Lemmas \ref{lem:elast2} and \ref{lem:high_freq} yields that for 
\begin{align}
\label{eq:functions1}
\begin{split}
\tilde{f}_{12}^{(1)}&:= \chi_{b_{12},\mu,\mu_1}(D) \tilde{\chi}_{11},\quad \tilde{g}_{12}^{(1)}:=  \chi_{b_{21},\mu,\mu_1}(D) \tilde{\chi}_{11}\\
\tilde{f}_{13}^{(1)}&:= \chi_{b_{13},\mu,\mu_1}(D) \tilde{\chi}_{11},\quad \tilde{g}_{13}^{(1)}:=  \chi_{b_{31},\mu,\mu_1}(D) \tilde{\chi}_{11},
\end{split}
\end{align}
by Plancherel, it holds that
\begin{align}
\label{eq:error_intro}
\begin{split}
& \|\tilde{\chi}_{11} - (\tilde{f}_{12}^{(1)} + \tilde{g}_{12}^{(1)} + \tilde{f}_{13}^{(1)} + \tilde{g}_{13}^{(1)})\|_{L^2}^2 \\
& \le C \sum_{\Z^3\setminus(\bigcup_{b\in\mathcal{B}_1}C_{b,\mu,\mu_1})}|\F\tilde\chi_{11}(k)|^2 +C|\F\tilde\chi_{11}(0)|^2 \\
&\le C\sum_{\bigcap_{b\in\mathcal{B}_1}\{|\hat k|^2-|\hat k\cdot b|^2\ge\mu^2\}}|\F\tilde\chi_{11}(k)|^2 + C\sum_{|k|\ge\mu_1}|\F\tilde\chi_{11}(k)|^2+C|\F\tilde\chi_{11}(0)|^2 \\
&\leq C(\mu^{-2} E_{el}(\chi) + \mu_1^{-1} E_{surf}(\chi)) \\
&\leq C(\mu^{-2}  + \mu_1^{-1} \epsilon^{-1}) E_{\epsilon}(\chi).
\end{split}
\end{align}
Here we used that
$|\hat{k}-b|\ge\frac{1}{2}|\hat{k}-(\hat{k}\cdot b)b|$ and we considered $\mu_0\in(0,1)$ a constant such that the sets $C_{b_{ij,2\mu,2\mu_1}}$ are all pairwise disjoint for every $\mu<\mu_0$.
Analogous decompositions and estimates also hold for the functions $\tilde{\chi}_{22}$ and $\tilde{\chi}_{33}$: More precisely, the functions
\begin{align}
\label{eq:functions2}
\begin{split}
&\tilde{f}_{12}^{(2)}:= \chi_{b_{12},\mu,\mu_1}(D) \tilde{\chi}_{22}, \quad \tilde{g}_{12}^{(2)}:=  \chi_{b_{21},\mu,\mu_1}(D) \tilde{\chi}_{22}, \\ &\tilde{f}_{23}^{(2)}:= \chi_{b_{23},\mu,\mu_1}(D) \tilde{\chi}_{22}, \quad \tilde{g}_{23}^{(2)}:=  \chi_{b_{32},\mu,\mu_1}(D) \tilde{\chi}_{22},\\
&\tilde{f}_{13}^{(3)}:= \chi_{b_{13},\mu,\mu_1}(D) \tilde{\chi}_{33}, \quad \tilde{g}_{31}^{(3)}:=  \chi_{b_{31},\mu,\mu_1}(D) \tilde{\chi}_{33}, \\
&\tilde{f}_{23}^{(3)}:= \chi_{b_{23},\mu,\mu_1}(D) \tilde{\chi}_{33}, \quad \tilde{g}_{32}^{(3)}:=  \chi_{b_{32},\mu,\mu_1}(D) \tilde{\chi}_{33},
\end{split}
\end{align}
satisfy
\begin{align*}
\begin{split}
\tilde{\chi}_{11} & = \tilde{f}_{12}^{(1)} + \tilde{g}_{12}^{(1)} + \tilde{f}_{13}^{(1)} + \tilde{g}_{13}^{(1)} + \tilde{\sigma}_{11},\\
\tilde{\chi}_{22} & = \tilde{f}_{12}^{(2)} + \tilde{g}_{12}^{(2)} + \tilde{f}_{23}^{(2)} + \tilde{g}_{23}^{(2)} + \tilde{\sigma}_{22},\\
\tilde{\chi}_{33} & = \tilde{f}_{13}^{(3)} + \tilde{g}_{13}^{(3)} + \tilde{f}_{23}^{(3)} + \tilde{g}_{23}^{(3)} + \tilde{\sigma}_{33},
\end{split}
\end{align*}
and, as in \eqref{eq:error_intro} above, 
\begin{align}
\label{eq:error_prelim}
\sum\limits_{j=1}^{3} \|\tilde{\sigma}_{jj}\|_{L^2}^2   \leq C(\mu^{-2}  + \mu_1^{-1} \epsilon^{-1}) E_{\epsilon}(\chi).
\end{align}
It remains to relate the functions in the decompositions for the phase indicators $\tilde{\chi}_{jj}$, i.e., to show that instead of the twelve functions from \eqref{eq:functions1}, \eqref{eq:functions2} it suffices to consider six functions as claimed in the statement.
To this end, we recall the trace constraint 
\begin{align*}
\tilde{\chi}_{11} + \tilde{\chi}_{22} + \tilde{\chi}_{33} = \eta_1 + \eta_2 + \eta_3 - \tr(\bar{e})=0,
\end{align*}
since $\tr(J_j)=\kappa+2\eta_1=\eta_1+\eta_2+\eta_3$ and $\bar{e} \in\mathcal{E}$.
In particular, applying $\chi_{b_{ij},\mu,\mu_1}(D)$ to the previous equation, we obtain that for $i,j \in \{1,2,3\}$, $i\neq j$,
\begin{align*}
\F\tilde{f}_{ij}^{(i)} +\F \tilde{f}_{ij}^{(j)} = -\chi_{b_{ij},\mu,\mu_1}\F\tilde{\sigma}_{\ell \ell} ,\quad
\F\tilde{g}_{ij}^{(i)} + \F\tilde{g}_{ij}^{(j)} = -\chi_{b_{ji},\mu,\mu_1} \F\tilde{\sigma}_{\ell \ell},
\end{align*}
for $\ell\neq i,j$. 
As a consequence, modifying the error terms $\tilde{\sigma}_{jj}$ appropriately, up to an error bound of the form \eqref{eq:error_prelim}, we may set for $i<j$
\begin{align}
\label{eq:functions3}
\F f_{ij} :=\F \tilde{f}_{ij}^{(i)}, \quad
\F g_{ij} :=\F \tilde{g}_{ij}^{(i)}, \quad \mbox{for } i,j \in \{1,2,3\},
\end{align}
and
\begin{align*}
\F\sigma_{11}&:= \F\tilde{\sigma}_{11},\\
\F \sigma_{22}&:= \F \tilde{\sigma}_{22} - \left(\chi_{b_{12},\mu,\mu_1} + \chi_{b_{21},\mu,\mu_1} \right)\F \tilde{\sigma}_{33},\\
\F \sigma_{33}&:= \F \tilde{\sigma}_{33} - \left(\chi_{b_{13},\mu,\mu_1} + \chi_{b_{31},\mu,\mu_1} \right)\F \tilde{\sigma}_{22}- \left(\chi_{b_{23},\mu,\mu_1}+ \chi_{b_{32},\mu,\mu_1} \right)
\F \tilde{\sigma}_{11},
\end{align*}
which yields the desired decomposition.
In particular, an energy estimate of the form \eqref{eq:error_prelim} remains valid.
\end{proof}

In order to prove the desired rigidity estimate of Theorem \ref{thm:rigidity}, we argue in parallel to the qualitative argument, and iteratively reduce the supports of the functions $f_{ij}$, $g_{ij}$ at a controlled energetic cost by using the structure of the wells. Here, in particular, a ``determinedness'' which results from the ellipticity of the problem will play a major role. We present the crucial step for such an iterative support reduction argument in the following section. 
As a preparation for this, in order to efficiently use the determinedness
property as in the stress-free setting, we will isolate certain directions in
frequency space. To this end, we consider finite differences and show that
energetically these allow to isolate only certain waves out of the six possible ones
from Lemma \ref{lem:decomp}. More precisely, for a fixed vector $v\in \R^3$ we
will denote the finite difference operator of step-size $h$ in direction $v \in \mathbb{S}^2$ by
\begin{align}
\label{eq:finite_diff}
\p_v^h u(x):= u(x + h v) - u(x).
\end{align}
If $v= e_j$ for some coordinate direction $e_j$, we also use the notation $\p_j^h$.

\begin{lem}[Finite difference control]
\label{lem:finite_diff}
Let $\chi\in BV(\T^3;K_3)$, $\bar{e}\in\mathcal{E}$ and $\tilde\chi=\chi-\bar{e}$, let $E_{el}(\chi)$ be as in \eqref{eq:el_chi} and let $b_{ij}$ and $\mathcal{B}$ be as in \eqref{eq:normals}.
Then there exist a constant $\mu_0=\mu_0(\mathcal{B})\in(0,1)$ and a universal constant $C>0$
such that for every $0<\mu<\mu_0$, $\tilde\mu>1$ and $h \in (0, \tilde\mu^{-1}\mu^{-1})$ and for $\ell,i,j\in\{1,2,3\}$, $\ell \neq i, j \neq \ell, j\neq i$,
\begin{align*}
\sum\limits_{k \in \Z^3 } |\mathcal{F}(\partial_{\ell}^h \chi_{b_{ij},\mu,\tilde\mu}(D)\tilde{\chi}_{ii})|^2  +
\sum\limits_{k \in \Z^3 } |\mathcal{F}(\partial_{\ell}^h \chi_{b_{ij},\mu,\tilde\mu}(D)\tilde{\chi}_{jj})|^2 
\leq C\mu^{-2}E_{el}(\chi).
\end{align*}
\end{lem}
\begin{proof}
We only prove the claim for the first term on the left hand side, the proof for the second one follows analogously.
We begin by observing that
\begin{align*}
\sum\limits_{k \in \Z^3 } |\mathcal{F}(\p_{\ell}^h \chi_{b_{ij}, \mu,\tilde\mu}(D) \tilde{\chi}_{ii}) |^2 
& = \sum\limits_{k \in \Z^3 } | (e^{2\pi ik_{\ell} h}-1) \chi_{b_{ij}, \mu,\tilde\mu}(k)\F\tilde{\chi}_{ii} |^2\\
& \leq C\sum\limits_{k \in \Z^3 } | k_{\ell} h \chi_{b_{ij}, \mu,\tilde\mu}(k)\F\tilde{\chi}_{ii}|^2 .
\end{align*}
In the last line, we used that $|1-e^{2\pi ik_\ell h}|^2=2-2\cos(2\pi k_\ell h)\le4\pi^2|k_\ell h|^2$ for every $k$ and $h$.
Relying on this bound, 
together with the condition that $|k|\leq 2\tilde\mu$,
$|\hat{k}_{\ell}|\leq|\hat{k}-b_{ij}|$ (since $b_{ij}\perp e_l$) and that $|\hat{k}-b_{ij}|\chi_{b_{ij},\mu,\tilde\mu}(k)\leq \dist(\hat{k}, \mathcal{B}_i)\chi_{b_{ij},\mu,\tilde\mu}(k)$ for sufficiently small $\mu_0$ (depending on $\mathcal{B}$), we hence obtain
\begin{align*}
\sum\limits_{k \in \Z^3 } | \mathcal{F}(\p_{\ell}^h \chi_{b_{ij}, \mu,\tilde\mu}(D) \tilde{\chi}_{ii} )|^2 
& \leq C\sum\limits_{k \in \Z^3 } | k_{\ell} h \chi_{b_{ij}, \mu,\tilde\mu}(k) \F\tilde{\chi}_{ii} |^2 \\
& \leq Ch^2 \tilde\mu^2 \sum\limits_{k \in \Z^3 } | \dist(\hat{k}, \mathcal{B}_i) \chi_{b_{ij}, \mu,\tilde\mu}(k) \F\tilde{\chi}_{ii} |^2 \\
&\leq C \mu^{-2} E_{el}(\chi),
\end{align*}
where we invoked Lemma \ref{lem:elast2} and the bound for $h$ in the last step.
\end{proof}

As a final ingredient for applying the ellipticity of the wells, i.e., the absence of rank-one connections, we make use of a nonlinear reduction argument which allows us to separate relevant from irrelevant nonlinear contributions.

\begin{lem}[Nonlinear reduction argument, I]
\label{cor:product}
Let $d_1 \in \N$, $A>0$ a universal constant and
for $\ell \in \{1,\dots,d_1\}$ let $G_{\ell}: \T^3 \rightarrow \R$ be given by $m_{\ell}(D) \tilde{\chi}_{j_\ell j_\ell}$ for some
$j_\ell\in\{1,2,3\}$ and some $m_{\ell}\in\mathcal{M}_A$ and $\tilde{\chi}$ as above.
Assume that $g: \T^3 \rightarrow \R$ satisfies the bound
\begin{align}
\label{eq:g_apriori}
\|g\|_{L^2(\T^3)}^2 \leq E \mbox{ and } \|g\|_{L^p(\T^3)} \leq C_1 \max\{p,(p-1)^{-1}\}^{18}
\end{align}
for some constants $E,C_1 >0$ and for all $p\in (1,\infty)$.
Then, for any $\gamma \in (0,1)$ there exists a constant $C(\gamma, d_1, C_1) \in (0, \frac{(Cd_1)^{2d_1}}{\gamma^{36 d_1}})$ such that for $r\in \{1,2,3\}$ and for every $h>0$ there holds
\begin{align*}
  &\int\limits_{\T^3}\left| \p_{r}^h g \prod\limits_{\ell=1}^{d_1} G_{\ell} \right|^2 dx  \leq C(\gamma, d_1,C_1,A) E^{1-\gamma}.
\end{align*}
\end{lem}

\begin{proof}
Without loss of generality, $r=2$. We apply Hölder's inequality with powers $1+\gamma$ and $\frac{1+\gamma}{\gamma}$ and afterwards $L^p$ interpolation. With this we obtain
\begin{align*}
\int\limits_{\T^3}|  \p_2^h g \prod\limits_{\ell=1}^{d_1} G_{\ell} |^2 dx
&\leq \| \p_2^h g\|_{L^{2+2\gamma}}^2 \Big\|\prod\limits_{\ell=1}^{d_1} G_{\ell} \Big\|_{L^{\frac{2+2\gamma}{\gamma}}}^2\\
& \lesssim \| \p_2^h g\|_{L^{2}}^{2(1-\gamma)} \| \p_2^h g\|_{L^{\frac{2+2\gamma}{\gamma}}}^{2\gamma} \Big\|\prod\limits_{\ell=1}^{d_1} G_{\ell} \Big\|_{L^{\frac{2+2\gamma}{\gamma}}}^2\\
& \lesssim E^{1-\gamma} \|   g\|_{L^{\frac{2+2\gamma}{\gamma}}}^{2\gamma} \Big\|\prod\limits_{\ell=1}^{d_1} G_{\ell} \Big\|_{L^{\frac{2+2\gamma}{\gamma}}}^2\\
& \leq C(\gamma,d_1,A) C_1 \max\{ \frac{2+2\gamma}{\gamma}, \frac{\gamma}{2+\gamma} \}^{36 \gamma} E ^{1-\gamma}.
\end{align*}
Here we have estimated the $L^2$ and $L^p$ norms by the assumption \eqref{eq:g_apriori} where we used the triangle inequality to estimate $\|\p_2^h g\|_{L^p} \leq 2 \|g\|_{L^p}$.
Moreover, we further estimate
\begin{align*}
\|   g\|_{L^{\frac{2+2\gamma}{\gamma}}}^{2\gamma} 
\leq  C_1 \max\{ \frac{2+2\gamma}{\gamma}, \frac{\gamma}{2+\gamma} \}^{36 \gamma}
\leq  \left(\frac{C_1 (2+2\gamma)}{\gamma} \right)^{36\gamma},
\end{align*}
the latter being bounded by a uniform constant as $\gamma \in (0,1)$.
For the contributions involving $G_{\ell}$, we applied Hölder's inequality and the assumed Marcinkiewicz bound
\begin{align*}
\Big\|\prod\limits_{\ell=1}^{d_1} G_{\ell} \Big\|_{L^{\frac{2+2\gamma}{\gamma}}}^2 \leq \prod\limits_{\ell=1}^{d_1}  \|G_{\ell} \|_{L^{\frac{2+2\gamma}{\gamma}d_1}}^2
\leq A \left(C \frac{2+2\gamma}{\gamma}d_1\right)^{36 d_1} \prod\limits_{\ell=1}^{d_1}  \|\tilde{\chi}_{j_\ell j_\ell}\|_{L^{\frac{2+2\gamma}{\gamma}d_1}}^2\lesssim  C(\gamma,d_1,A).
\end{align*}
\end{proof}

In particular, we record the following variant of Lemma \ref{cor:product}.

\begin{lem}[Nonlinear reduction argument, II]
\label{eq:product_nonline}
Let $\chi\in BV(\T^3;K_3)$, $\bar{e}\in\mathcal{E}$ and $\tilde\chi=\chi-\bar{e}$, let $E_{el}(\chi)$ be as in \eqref{eq:el_chi} and let $b_{ij}$ and $\mathcal{B}$ be as in \eqref{eq:normals}.
Let $\frac{1}{2}\le\tilde\mu$, $0<\mu<\mu_0$ with $\mu_0$ as in the statement of Lemma \ref{lem:finite_diff}.
Let $i,j\in\{1,2,3\}$, $i< j$, 
and let $f_{ij}, g_{ij}$ denote the functions from Lemma \ref{lem:decomp}.
Let $d_1 \in \N$, $A>0$ be a universal constant and for $\ell \in \{1,\dots,d_1\}$ let $G_{\ell}: \T^3 \rightarrow \R$ be given by $m_{\ell}(D) \tilde{\chi}_{j_\ell j_\ell}$ for some $j_\ell\in\{1,2,3\}$ and some
$m_\ell\in\mathcal{M}_A$.
Then, for any $\gamma \in (0,1)$ there exists a constant $C(\gamma, d_1) \in (0, \frac{(Cd_1)^{2d_1}}{\gamma^{36 d_1}})$ such that for $r\in \{1,2,3\}\setminus\{i,j\}$
\begin{align*}
\int\limits_{\T^3}\left| \p_{r}^h \chi_{b_{ij},\mu,\tilde\mu}(D) f_{ij} \prod\limits_{\ell=1}^{d_1} G_{\ell} \right|^2 dx + \int\limits_{\T^3}\left| \p_{r}^h \chi_{b_{ji},\mu,\tilde\mu}(D) g_{ij} \prod\limits_{\ell=1}^{d_1} G_{\ell} \right|^2 dx \leq C(\gamma, d_1, A) (\mu^{-2}  E_{el}(\chi))^{1-\gamma}.
\end{align*}
\end{lem}

In Section \ref{sec:lower}, we will apply this lemma with $\prod_{\ell} G_{\ell}$
corresponding to a polynomial in terms of different frequency-localizations of $\tilde{\chi}_{ii}$.

\begin{proof}
We apply Lemma \ref{cor:product} with $g = \p_r^h\chi_{b_{ij},\mu,\tilde\mu}(D) f_{ij} $ and $g = \p_{r}^h \chi_{b_{ji},\mu,\tilde\mu}(D) g_{ij}$. The $L^2$ estimates follow from Lemma \ref{lem:finite_diff}.
The $L^p$ bounds, in turn, are a consequence of the fact that the multipliers $\chi_{b_{ij},\mu,\tilde\mu}(k)$ satisfy the assumptions of the Marciniewicz multiplier theorem, see \eqref{eq:marcinkiewicz}.
\end{proof}

\section{Proof of the Lower Bound in Theorems \ref{thm:rigidity} and \ref{thm:scaling}}
\label{sec:lower}

\subsection{The induction argument}

The main objective of this section is to deduce the rigidity estimate of Theorem \ref{thm:rigidity} and, building on this, the lower bound estimate from Theorem \ref{thm:scaling}.
To this end, we iteratively reduce the support of the functions from the decomposition in Lemma \ref{lem:decomp}. Compared to the argument from \cite{RT22} this involves substantial additional difficulties due to the presence of the six planar waves in the decomposition in Lemma \ref{lem:decomp}. This does not allow for an immediate ``determinedness'' argument as in \cite{RT22}. Instead we make use of the preliminary considerations from the previous section to isolate the key contributions and then to iteratively apply a determinedness result involving the nonlinear, elliptic structure of the wells.

In what follows, 
for any $\frac{1}{2}\le\tilde\mu\le\mu_1$ and $b_{ij}\in\mathcal{B}$
we define 
$\chi_{b_{ij}, \mu, [\tilde\mu, \mu_1]}$ to be equal to the function $\chi_{b_{ij}, \mu, \mu_1}\chi_{A_{\tilde\mu,\mu_1}}$ where $\chi_{A_{\tilde\mu,\mu_1}}$ denotes a smoothed out version of the characteristic function of the annulus $A_{\tilde\mu, \mu_1}:=\{k \in \R^3: \ \tilde\mu \leq |k| \leq \mu_1\}$ which equals one in $A_{\tilde\mu, \mu_1}$ and with support in the annulus $A_{\frac{\tilde\mu}{2}, 2\mu_1}$.
We also associate the Fourier multiplier $\chi_{b_{ij}, \mu, [\tilde{\mu}, \mu_1]}(D) $ with it.
As for the truncated cones from above, with slight abuse of notation, we view these annuli and their characteristic functions also as functions on $\Z^3$ without commenting on this specifically in what follows.
Using this notation, we obtain a first reduction statement for the support of
the functions from the decomposition in Lemma \ref{lem:decomp}.

In what follows, a key role will be played by the ellipticity condition \eqref{eq:elliptic}.
Indeed, this condition allows us to find polynomials $P_{ij \ell}: \R \rightarrow \R$ such that 
\begin{equation}\label{eq:poly}
\p_\ell^{h}\chi_{jj} = P_{ij \ell}(\p_{\ell}^{h} \chi_{ii})
\quad \text{for all } h\in \R.
\end{equation}

This corresponds to a ``determinedness condition''.
Indeed, the ellipticity condition \eqref{eq:elliptic} allows us to conclude that these polynomials exist and are independent of $h \in \R$.
In what follows, we will denote the degree of these polynomials by $d\in \N$. Since $\p_{\ell}^{h}\chi_{jj}$ attains only seven possible values (see also the explanations in \eqref{eq:welldifferences} below), we note that it is possible to choose $d = 6$.

\begin{lem}[Reduction of the Fourier support]
\label{lem:induction_basis}
Let $\ell, i \in\{1,2,3\}$, $\ell < i$. Let $m\in\N$, $m\ge1$, $\gamma\in(0,1)$ and let $d$ denote the maximal degree of the polynomials from \eqref{eq:poly}.
Let $E_{\epsilon}(\tilde{\chi})$ be as in \eqref{eq:e-eps}. 

Let $0<\mu<\frac{1}{2}<\mu_{m}<\mu_{1}$ and $\mu_j = (M \mu)^{j} \mu_1$ for some $M \in \N$ such that $M\mu\in(0,1)$. 
Assume that
\begin{align}
\label{eq:functions-m}
\begin{split}
\tilde{\chi}_{11} & = f_{12,m} + g_{12,m} + f_{13,m} + g_{13,m} + \sigma_{11,m},\\
\tilde{\chi}_{22} & = -f_{12,m} - g_{12,m} + f_{23,m} + g_{23,m} + \sigma_{22,m},\\
\tilde{\chi}_{33} & = -f_{13,m} - g_{13,m} - f_{23,m} - g_{23,m} + \sigma_{33,m},
\end{split}
\end{align}
with $f_{ij,m}:= \chi_{b_{ij},\mu,\mu_{m}}(D)f_{ij}$, $g_{ij,m}:= \chi_{b_{ji}, \mu, \mu_{m}}(D)g_{ij}$, and
\begin{align}
\label{eq:ind-hyp}
\sum\limits_{j=1}^{3} \int\limits_{\T^3} |\sigma_{jj,m}|^2 dx \leq \tilde{C}_{\gamma}^m \max\{ (\mu^{-2} + \mu_1^{-1} \epsilon^{-1}) E_{\epsilon}(\chi), (\mu^{-2} + \mu_1^{-1} \epsilon^{-1})^{(1-\gamma)^m} E_{\epsilon}(\chi)^{(1-\gamma)^m} \},
\end{align}
with $\tilde{C}_{\gamma} = \frac{(\tilde C d)^{2d}}{\gamma^{36d}}$.

Then, if $\tilde{C}>1$ in the definition of $\tilde{C}_{\gamma}$ is sufficiently large, it holds that
\begin{align}
\label{eq:induc}
\begin{split}
&\|\chi_{b_{\ell i}, \mu,  [\mu_{m+1}, \mu_m]}(D) f_{\ell i,m}\|^2_{L^2(\T^3)} 
+ \|\chi_{b_{i \ell}, \mu, [\mu_{m+1}, \mu_m]}(D) g_{\ell i,m}\|^2_{L^2(\T^3)} \\
&\leq \tilde{C}_{\gamma}^{m+1} \max\{ (\mu^{-2} + \mu_1^{-1} \epsilon^{-1}) E_{\epsilon}(\tilde{\chi}), (\mu^{-2} + \mu_1^{-1} \epsilon^{-1})^{(1-\gamma)^{m+1}} E_{\epsilon}(\tilde{\chi})^{(1-\gamma)^{m+1}} \}.
\end{split}
\end{align}
\end{lem}

This lemma will play a central role in reducing the Fourier supports of the six plane waves from Lemma \ref{lem:decomp} at a controlled energetic cost. The energetic cost will be quantified in terms of the constant $\tilde{C}_{\gamma}$ and in the (possible) loss in the exponent in the energy in \eqref{eq:induc}.

\begin{proof}
We consider the case in which $\ell = 2, i=3$ and prove the claimed localization
estimates on the level of the function $c_{23,m}:= -f_{23,m} - g_{23,m}$. 
To shorten the notation, we will also write $c_{12,m}:= f_{12,m} + g_{12,m}$ and $c_{13,m}:= -f_{13,m} -g_{13,m}$.
The other cases follow analogously.

\emph{Step 1: Finite difference control.}
Denoting the finite difference operator of step-size $h>0$ in direction $2$ by $\p_2^h$ (see \eqref{eq:finite_diff}),
and choosing $h \in (0,\mu_m^{-1}\mu^{-1})$, by hypothesis \eqref{eq:ind-hyp} and applying
Lemmas \ref{lem:decomp} and \ref{lem:finite_diff} we obtain
\begin{align*}
&\sum\limits_{k\in\Z^3} |\F \p_2^{h} \tilde{\chi}_{33} -  \F \p_2^{h} c_{23,m}|^2 \leq 4 \sum\limits_{k\in\Z^3} | \F \p_2^h \chi_{b_{13}, \mu,\mu_m}(D) \tilde{\chi}_{33} |^2 \\
&\qquad + 4 \sum\limits_{k\in\Z^3} | \F \p_2^h \chi_{b_{31}, \mu,\mu_m}(D) \tilde{\chi}_{33} |^2 + 4 \sum\limits_{k\in\Z^3} |\F\p_2^h\sigma_{33,m}|^2 \\
&\qquad + 4 \sum_{k\in\Z^3} |\F((\chi_{b_{13},\mu,\mu_m}(D)+\chi_{b_{31},\mu,\mu_m}(D))\p_2^h\sigma_{22})|^2 \\
& \leq 8C \mu^{-2}   E_{el}(\chi) +  8\tilde{C}_{\gamma}^m \max\{ (\mu^{-2} + \mu_1^{-1} \epsilon^{-1}) E_{\epsilon}(\chi), (\mu^{-2} + \mu_1^{-1} \epsilon^{-1})^{(1-\gamma)^m} E_{\epsilon}(\chi)^{(1-\gamma)^m} \}\\
& \qquad +8C(\mu^{-2}+\epsilon^{-1}\mu_1^{-1})E_\epsilon(\chi) \\
&\leq (16C+8\tilde{C}_\gamma^m) \max\{ (\mu^{-2} + \mu_1^{-1} \epsilon^{-1}) E_{\epsilon}(\chi), (\mu^{-2} + \mu_1^{-1} \epsilon^{-1})^{(1-\gamma)^m} E_{\epsilon}(\chi)^{(1-\gamma)^m} \}.
\end{align*}
In the first inequality above we also used that, by \eqref{eq:functions2}, $\chi_{b_{13},\mu,\mu_m}(D)f_{13,m}=\chi_{b_{13},\mu,\mu_m}(D)\tilde{\chi}_{33}+\chi_{b_{13}\mu,\mu_m}(D)\sigma_{22}$, and that the analogous identity holds for $g_{13,m}$ as well.

\medskip

\emph{Step 2: Isolating relevant and irrelevant contributions.}
By determinedness 
there exists a polynomial $P_{132}$ as in \eqref{eq:poly}, i.e.,\ of degree $d\geq 1$ such that $\p_2^h \tilde{\chi}_{33} = P_{132}(\p_2^h \tilde{\chi}_{11})$. We emphasize that this polynomial does \emph{not} depend on $h \in \R$. Inserting this into the estimate from Step 1 hence yields the following bound
\begin{align}\label{eq:nonlin-est}
\begin{split}
\sum \limits_{k\in\Z^3} |\F P_{132}(\p_2^{h} \tilde{\chi}_{11}) -  \F \p_2^{h} c_{23,m}|^2 
&\le (16C+8\tilde{C}_\gamma^m) \max\{ (\mu^{-2} + \mu_1^{-1} \epsilon^{-1}) E_{\epsilon}(\chi), \\
&\qquad (\mu^{-2} + \mu_1^{-1} \epsilon^{-1})^{(1-\gamma)^m} E_{\epsilon}(\chi)^{(1-\gamma)^m} \}.
\end{split}
\end{align} 
We next seek to use the decomposition for $\tilde{\chi}_{11}$ (see \eqref{eq:functions-m}) together with the fact that, by the conical localization, there is good control on $\p_2^h c_{13,m}$. To this end, we note that by the decomposition in \eqref{eq:functions-m}, the polynomial $P_{132}(\p_2^{h} \tilde{\chi}_{11}) $ can be expanded into polynomials in $\p_2^h c_{12,m}$, $\p_2^h c_{13,m}$ and $\p_2^h \sigma_{11,m}$ including all their mixed terms: Recalling that $d$ denotes the degree of $P_{132}$, and thanks to \eqref{eq:nonlin-est} by the triangle inequality, we have
\begin{align*}
&\sum \limits_{k\in\Z^3} |\F P_{132}(\p_2^{h} c_{12,m}) -  \F \p_2^{h} c_{23,m}|^2 
 \lesssim\sum \limits_{k\in\Z^3} |\F P_{132}(\p_2^{h} \tilde\chi_{11}) -  \F \p_2^{h} c_{23,m}|^2  \\
&\quad\qquad  + \sum\limits_{k\in\Z^3} |\F P_{132}(\p_2^{h} c_{12,m}) - \F P_{132}(\p_2^{h} \tilde\chi_{11})|^2  \\
& \quad \lesssim (16C+8\tilde{C}_\gamma^m) \max\{ (\mu^{-2} + \mu_1^{-1} \epsilon^{-1}) E_{\epsilon}(\chi),(\mu^{-2} + \mu_1^{-1} \epsilon^{-1})^{(1-\gamma)^m} E_{\epsilon}(\chi)^{(1-\gamma)^m} \} \\
& \quad\qquad + C_d \sum_{d_1=1}^{d}\sum_{j'=1}^{d_1}\sum_{j=1}^{d_1-j+1} \int_{\T^3}|f_j \prod_{\ell=1}^{d_1-1} G_{\ell}^{(j,j',d_1)}|^2 ,
\end{align*}
where $f_j \in \{\p_2^h c_{13,m}, \p_2^h \sigma_{11,m}\}$ and $G_{\ell}^{(j,j',d_1)} \in \{\p_2^h c_{13,m}, \p_2^h \sigma_{11,m},\p_2^h c_{12,m} \}$. 
Here we used that $P_{132}(\p_2^{h} \tilde\chi_{11}) =  P_{132}(\p_2^{h} c_{12,m}) + P$, where $P$ is a polynomial of $f_j$, $G_{\ell}^{(j,j',d_1)}$ of the type $f_j \prod\limits_{\ell =1}^{d-j'} G_{\ell}^{(j,j',d_1)}$.
We claim that all monomials involving at least one factor of either $\p_2^h c_{13,m}$ or $\p_2^h \sigma_{11,m}$ (i.e., as those in the second line of the inequality above) can be bounded (with slight loss due to Calder\'on Zygmund estimates).

Indeed, this follows from Lemma \ref{cor:product} and (for the case $m=1$) Lemma \ref{eq:product_nonline}. We consider, as an example a product which involves $\p_2^h c_{13,m}$ at least linearly. Then, recalling that $h \in (0,\mu_m^{-1}\mu^{-1})$ by Lemma \ref{lem:finite_diff} together with Lemma \ref{cor:product} (or Lemma \ref{eq:product_nonline} if $m=1$), there exists $C(\gamma, d) \in (0, \frac{(Cd)^{2d}}{\gamma^{36d}})$, such that
\begin{align*}
\int\limits_{\T^3}| \p_2^h c_{13,m} \prod\limits_{\ell=1}^{d - j'} G_{\ell}^{(j,j',d_1)}|^2
& \leq C(\gamma,d)  \tilde{C}_{\gamma}^m \max\{ (\mu^{-2} + \mu_1^{-1} \epsilon^{-1}) E_{\epsilon}(\chi), \\
& \qquad \qquad (\mu^{-2} + \mu_1^{-1} \epsilon^{-1})^{(1-\gamma)^{m+1 } } E_{\epsilon}(\chi)^{(1-\gamma)^{m+1 }} \}.
\end{align*}
Notice that, in applying Lemmas \ref{lem:finite_diff} and \ref{eq:product_nonline} we used the fact that $c_{13,m}=(\chi_{b_{13},\mu,\mu_m}(D)+\chi_{b_{31},\mu,\mu_m}(D))\tilde\chi_{11}$. 
If $\p_2^h \sigma_{11,m}$ appears at least linearly,  we argue analogously by invoking Lemma \ref{cor:product} and \eqref{eq:ind-hyp}.
As a consequence, we obtain
\begin{align}
\label{eq:monomial_reduc}
\begin{split}
&\sum\limits_{k\in\Z^3}|\F P_{132}(\p_2^{h} c_{12,m}) -  \F \p_2^{h} c_{23,m}|^2 \\
& \quad \leq (16C+8\tilde{C}_{\gamma}^m+\tilde C\tilde{C}_{\gamma}^m) \max\{ (\mu^{-2} + \mu_1^{-1} \epsilon^{-1}) E_{\epsilon}(\chi),(\mu^{-2} + \mu_1^{-1} \epsilon^{-1})^{(1-\gamma)^{m+1}} E_{\epsilon}(\chi)^{(1-\gamma)^{m+1}} \},
\end{split}
\end{align}
where $\tilde C>0$ is a constant depending on the polynomials $P_{ijk}$ from \eqref{eq:poly}, and thus in turn on $K_3$.

\emph{Step 3: Determinedness, support reduction for $\p_2^h c_{23,m}$.}
With \eqref{eq:monomial_reduc} available, in order to infer information on the
Fourier support of $\p_2^h c_{23,m}$, we next study the Fourier support of
$P_{132}(\p_2^h c_{12,m})$. As a key observation we will note that since the 
Fourier supports of $c_{12,m}$ and $c_{23,m}$ are disjoint, the Fourier supports of $P_{132}(\p_2^h
c_{12,m})$ and $\p_2^h c_{23,m}$ 
intersect in a (relatively) small set, whose measure is controlled and quantified.
Hence, by a similar argument as in the scaling result for the Tartar square \cite{RT22}, one
can reduce the Fourier support of $c_{23,m}$.

We present the details of this argument: 
Firstly, we observe that $P_{132}(\p_2^h c_{12,m})$ has Fourier support in a (truncated) conical neighbourhood of the plane spanned by the vectors $b_{12}, b_{21}$, while $\p_2^hc_{23,m}$ has support in a truncated neighbourhood of the lines spanned by $b_{23}, b_{32}$, respectively:
\begin{align}
\label{eq:supp_disjoint1}
\begin{split}
&\supp(\F \p_2^h c_{23,m}) \subset C_{b_{23},2\mu,2\mu_{m}}\cup C_{b_{32},2\mu,2\mu_m}, \\ 
&\supp(\F P_{132}(\p_2^h c_{12,m}))\subset 
\left\{ k \in \R^3: \  |k_3| \leq M \mu \mu_m= \mu_{m+1}, \ |k|\leq M \mu_m \right\},
\end{split}
\end{align}
where $M \in \N$ depends only on the degree of $P_{132}$ and on $\mathcal{B}$.
In particular, we observe that
\begin{align*}
\supp(\F P_{132}(\p_2^h c_{12,m})) \cap  \supp(\chi_{ \{ |k_3| \in [\mu_{m+1},\mu_m] \}} \F \p_2^h c_{23,m}) = \emptyset.
\end{align*}
Hence, from \eqref{eq:supp_disjoint1} and \eqref{eq:monomial_reduc}, 
we conclude that
\begin{align}
\label{eq:support_deriv}
\begin{split}
&\sum\limits_{k\in\Z^3} | \chi_{ \{ |k_3| \in [\mu_{m+1},\mu_m] \}}(k) \F \p_2^h c_{23,m}|^2 \\
& \quad \leq (16C+8\tilde{C}_{\gamma}^m+\tilde C\tilde{C}_{\gamma}^m) \max\{ (\mu^{-2} + \mu_1^{-1} \epsilon^{-1}) E_{\epsilon}(\chi),(\mu^{-2} + \mu_1^{-1} \epsilon^{-1})^{(1-\gamma)^{m+1}} E_{\epsilon}(\chi)^{(1-\gamma)^{m+1}} \}.
\end{split}
\end{align}

\emph{Step 4: Support reduction for $c_{23,m}$.}
We claim that we can further use \eqref{eq:support_deriv} to reduce the support of $c_{23,m}$. 
Indeed, we first note that
\begin{align*}
\sum\limits_{k\in\Z^3} |  \chi_{ \{ |k_3| \in [\mu_{m+1},\mu_m] \}}(k) \F \p_2^h c_{23,m}|^2 
= \sum\limits_{k\in\Z^3} |  \chi_{ \{ |k_3| \in [\mu_{m+1},\mu_m] \}}(k)(e^{2\pi i h k_2}-1) \F c_{23,m}|^2 .
\end{align*}
We thus need to show that $|e^{2\pi i h k_2}-1|$ is, on average, bounded away from zero.
Indeed, by the definition of the conical localization \eqref{eq:Cbij} in the Fourier set under consideration $(C_{b_{23},\mu,\mu_m}\cup C_{b_{32},\mu,\mu_m}) \cap \{|k_3| \in [\mu_{m+1}, \mu_m]\}$, it holds that for $0<\mu<1$ small, $k/|k|$ is ``almost" parallel to $b_{23}=\frac{1}{\sqrt{2}}(e_2+e_3)$ or $b_{32}=\frac{1}{\sqrt{2}}(-e_2+e_3)$,
\begin{align*}
|k_2|\sim|k_3|\sim \frac{1}{\sqrt{2}}|k| \in [\frac{1}{\sqrt{2}}\mu_{m+1},\frac{1}{\sqrt{2}}\mu_m].
\end{align*}
To be precise, straightforward calculations yield, $\frac{1}{4}|k|\le|k_2|\le|k|$.
So, $|k_2|\in[\frac{\mu_{m+1}}{4},\mu_m]$.
In particular, for the choice $h\in (\frac{1}{10\mu \mu_m},
\frac{1}{\mu \mu_m})\subset (0,\frac{1}{\mu \mu_m})$
it follows that
\begin{align*}
h|k_2|\in\Big[\frac{M}{40},\frac{1}{\mu}\Big]
\end{align*}
which is at least of the order one.
Notice that, by the condition $M\mu_0<1$, the interval above is nontrivial and has length larger than $1$ for $\mu$ sufficiently small.
Thus, for $h$ in this part of $(0,\frac{1}{\mu \mu_m})$ the product $h |k_2|$ is bounded away from zero and large compared to the period length of $e^{i \cdot}$. Hence, computing an average in $h$, we obtain
\begin{align*}
&\mu \mu_m \int_{\frac{1}{10\mu \mu_m}}^{\frac{1}{\mu \mu_m}} \sum\limits_{k\in\Z^3} |  \chi_{ \{ |k_3| \in [\mu_{m+1},\mu_m] \}}(k) \F \p_2^h c_{23,m}|^2 \ dh\\
&\qquad =   \sum\limits_{k\in\Z^3} |  \chi_{ \{ |k_3| \in [\mu_{m+1},\mu_m] \}}(k) \F  c_{23,m}|^2  \fint\limits_{(\frac{1}{10\mu \mu_m},\frac{1}{\mu \mu_m})} |e^{i h k_2}-1|^2 dh \\
&\qquad \gtrsim  \sum\limits_{k\in\Z^3} |  \chi_{ \{ |k_3| \in [\mu_{m+1},\mu_m] \}}(k) \F c_{23,m}|^2.
\end{align*}
We note that by the choice of $h$ all constants can be chosen to be independent of $h,\mu_m, \mu_{m+1},\mu$.

Therefore, applying an $h$-average to the estimate
  \eqref{eq:support_deriv}, we deduce that
\begin{align}
\label{eq:directional_reduc2}
\begin{split}
&\sum\limits_{k\in\Z^3} | \chi_{ \{ |k_3| \in [\mu_{m+1},\mu_m] \}}(k) \F c_{23,m}|^2 \\
&\quad \leq (16C+8\tilde{C}_{\gamma}^m+\tilde C\tilde{C}_{\gamma}^m) \max\{ (\mu^{-2} + \mu_1^{-1} \epsilon^{-1}) E_{\epsilon}(\chi),(\mu^{-2} + \mu_1^{-1} \epsilon^{-1})^{(1-\gamma)^{m+1}} E_{\epsilon}(\chi)^{(1-\gamma)^{m+1}} \}.
\end{split}
\end{align}

\emph{Step 5: Conclusion.}
Carrying out the same argument with $\p_3^h c_{23,m}$ instead of $\p_2^h c_{23,m}$, one obtains that
\begin{align}
\label{eq:directional_reduc3}
\begin{split}
&\sum\limits_{k\in\Z^3} | \chi_{ \{ |k_2| \in [\mu_{m+1},\mu_m] \}}(k) \F c_{23,m}|^2 \\
&\quad \leq (16C+8\tilde{C}_{\gamma}^m+\tilde C\tilde{C}_{\gamma}^m) \max\{ (\mu^{-2} + \mu_1^{-1} \epsilon^{-1}) E_{\epsilon}(\chi),(\mu^{-2} + \mu_1^{-1} \epsilon^{-1})^{(1-\gamma)^{m+1}} E_{\epsilon}(\chi)^{(1-\gamma)^{m+1}} \}.
\end{split}
\end{align}
Given this control of $c_{23,m}$, we next turn to estimating the individual waves $f_{23,m}$ and $g_{23,m}$. These satisfy $c_{23,m}=-f_{23,m}-g_{23,m}$. To this end, we recall that the Fourier supports of $f_{23,m}$ and $g_{23,m}$ are contained in cones with orthogonal center directions $b_{23}$ and $b_{32}$, respectively (see Figure \ref{fig:conicallocalization}).
The intersections of these cones with the annulus with radii $\mu_{m+1}, \mu_{m}$ are disjoint and therefore $\chi_{b_{23}, \mu, [\mu_{m+1}, \mu_m]}(k) c_{23,m}$ provides control of both $\chi_{b_{23}, \mu, [\mu_{m+1}, \mu_m]}(k) \F f_{23,m}$ and $\chi_{b_{32}, \mu, [\mu_{m+1}, \mu_m]}(k) \F g_{23,m}$ individually.
Combining \eqref{eq:directional_reduc2} and \eqref{eq:directional_reduc3}, we deduce that
 \begin{align}
\label{eq:directional_reduc_combi}
\begin{split}
&\sum\limits_{k\in\Z^3} | \chi_{b_{23}, \mu, [\mu_{m+1}, \mu_m]}(k) \F f_{23,m}|^2 
+ \sum\limits_{k\in\Z^3} | \chi_{b_{32}, \mu, [\mu_{m+1}, \mu_m]}(k) \F g_{23,m}|^2 \\ 
&\, \leq 2 \sum\limits_{k\in\Z^3} | \chi_{ \{ |k_3| \in [\mu_{m+1},\mu_m] \}}(k)\chi_{ \{ |k_2| \in [\mu_{m+1},\mu_m] \}}(k) \F c_{23,m}|^2 \\
&\, \leq 2(16C+8\tilde{C}_{\gamma}^m+\tilde C\tilde{C}_{\gamma}^m) \max\{ (\mu^{-2} + \mu_1^{-1} \epsilon^{-1}) E_{\epsilon}(\chi),(\mu^{-2} + \mu_1^{-1} \epsilon^{-1})^{(1-\gamma)^{m+1}} E_{\epsilon}(\chi)^{(1-\gamma)^{m+1}} \}.
\end{split}
\end{align}
Finally, choosing the constant $\tilde C$ in the definition of $\tilde{C}_{\gamma}>1$ such that $2(16C+8\tilde{C}_{\gamma}^m+\tilde C\tilde{C}_{\gamma}^m)\le\tilde C_\gamma^{m+1}$, e.g.,\ $\tilde C\ge 68C+36$, this hence concludes the proof of the proposition.
\end{proof}

As a consequence of the previous support reduction argument, we can upgrade the plane wave decomposition from Lemma \ref{lem:decomp}. We obtain an improved decomposition with reduced Fourier support of the main contributions at a controlled energetic cost.

\begin{lem}[Induction step]
\label{lem:inductive_step}
Let $0<\mu< \frac{1}{2} <\mu_{m}<\mu_{1}$ and $\mu_j = (M \mu)^{j} \mu_1$ for some $M \in \N$ such that $M\mu<1$. Let $d\in \N$ be the degree of the polynomials from \eqref{eq:poly} in the determinedness condition.
Assume that
\begin{align*}
\tilde{\chi}_{11} & = f_{12,m} + g_{12,m} + f_{13,m} + g_{13,m} + \sigma_{11,m},\\
\tilde{\chi}_{22} & = -f_{12,m} - g_{12,m} + f_{23,m} + g_{23,m} + \sigma_{22,m},\\
\tilde{\chi}_{33} & = -f_{13,m} - g_{13,m} - f_{23,m} - g_{23,m} + \sigma_{33,m},
\end{align*}
with $f_{ij,m}:= \chi_{b_{ij},\mu,\mu_{m}}(D)f_{ij}$, $g_{ij,m}:= \chi_{b_{ji}, \mu, \mu_{m}}(D)g_{ij}$, 
and
\begin{align*}
\sum\limits_{j=1}^{3} \int\limits_{\T^3} |\sigma_{jj,m}|^2 dx \leq \bar{C}_{\gamma}^m \max\{ (\mu^{-2} + \mu_1^{-1} \epsilon^{-1}) E_{\epsilon}(\chi), (\mu^{-2} + \mu_1^{-1} \epsilon^{-1})^{(1-\gamma)^m} E_{\epsilon}(\chi)^{(1-\gamma)^m} \},
\end{align*}
with $\bar{C}_{\gamma} = \frac{(C d)^{2d}}{\gamma^{36d}}$.
Then, if the constant $C>1$ in the definition of $\bar{C}_{\gamma}$ is sufficiently large, there is a decomposition 
\begin{align}\label{eq:dec-final}
\begin{split}
\tilde{\chi}_{11} & = f_{12,m+1} + g_{12,m+1} + f_{13,m+1} + g_{13,m+1} + \sigma_{11,m+1},\\
\tilde{\chi}_{22} & = -f_{12,m+1} - g_{12,m+1} + f_{23,m+1} + g_{23,m+1} + \sigma_{22,m+1},\\
\tilde{\chi}_{33} & = -f_{13,m+1} - g_{13,m+1} - f_{23,m+1} - g_{23,m+1} + \sigma_{33,m+1},
\end{split}
\end{align}
with 
$f_{ij,m+1} = \chi_{b_{ij}, \mu, \mu_{m+1}}(D) f_{ij},\ g_{ij,m+1}= \chi_{b_{ji}, \mu, \mu_{m+1}}(D) g_{ij}$, 
and
\begin{align}\label{eq:est-final}
\begin{split}
\sum\limits_{j=1}^{3} \int\limits_{\T^3} |\sigma_{jj,m+1}|^2 dx 
& \leq \bar{C}_{\gamma}^{m+1} \max\{ (\mu^{-2} + \mu_1^{-1} \epsilon^{-1}) E_{\epsilon}(\chi),\\
& \quad \quad (\mu^{-2} + \mu_1^{-1} \epsilon^{-1})^{(1-\gamma)^{m+1}} E_{\epsilon}(\chi)^{(1-\gamma)^{m+1}} \}.
\end{split}
\end{align}
\end{lem}

\begin{proof}
This follows immediately from the previous lemma. Indeed, considering $\tilde{\chi}_{11}$, this follows by, for instance, setting 
\begin{align*}
f_{12,m+1}&:= \chi_{b_{12},\mu,\mu_{m+1}}(D) f_{12}, \ g_{12,m+1}:= \chi_{b_{21},\mu,\mu_{m+1}}(D) g_{12}, \\
f_{13,m+1}&:= \chi_{b_{13},\mu,\mu_{m+1}}(D) f_{13}, \
g_{13,m+1}:= \chi_{b_{31},\mu,\mu_{m+1}}(D) g_{13},\\
 \sigma_{11,m+1}&:= \sigma_{11,m} + \big(\chi_{b_{12},\mu,\mu_m}(D)-\chi_{b_{12},\mu,\mu_{m+1}}(D)\big) f_{12} + \big(\chi_{b_{21},\mu,\mu_m}(D)-\chi_{b_{21},\mu,\mu_{m+1}}(D)\big) g_{12} \\
& \quad  + \big(\chi_{b_{13},\mu,\mu_m}(D)-\chi_{b_{13},\mu,\mu_{m+1}}(D)\big) f_{13}+ \big(\chi_{b_{31},\mu,\mu_m}(D)-\chi_{b_{31},\mu,\mu_{m+1}}(D)\big) g_{13}. 
\end{align*}
Analogous formulas can be set for the other components.
Noticing that $0\le\chi_{b,\mu,\mu_m}-\chi_{b,\mu,\mu_{m+1}}\le\chi_{b,\mu,[\mu_{m+1},\mu_m]}$ for any $b\in\mathcal{B}$, by Lemma \ref{lem:decomp} and Lemma \ref{lem:induction_basis}, this 
implies the desired decomposition and error estimates.
\end{proof}

\subsection{Proofs of the rigidity estimates of Theorems \ref{thm:rigidity} and \ref{thm:scaling}}
\label{sec:thms_proof}

With the inductive Fourier decomposition of Lemma \ref{lem:inductive_step} in place, we proceed to the proofs of Theorems \ref{thm:rigidity} and \ref{thm:scaling}. The argument here closely resembles the one in the setting of the Tartar square. We provide the details for the convenience of the reader.

\subsubsection{Proof of the lower bound in Theorem \ref{thm:rigidity}}

We first address the quantitative rigidity estimate from Theorem \ref{thm:rigidity}. The scaling result of Theorem \ref{thm:scaling} will then follow with only small modifications in the next section.

\begin{proof}[Proof of Theorem \ref{thm:rigidity}]
As in \cite{RT22} the lower bound in Theorem \ref{thm:rigidity} follows from the error estimate in Lemma \ref{lem:inductive_step} by an optimization argument. 
For the convenience of the reader, we present the arguments in detail.

We preliminary notice that, by finite induction, by Lemma \ref{lem:decomp} (being the induction base) and Lemma \ref{lem:inductive_step}, for every $m$ such that $\mu_m\ge \frac{1}{2}$, \eqref{eq:dec-final} and \eqref{eq:est-final} hold true.
We also notice that, without loss of generality we can assume that $E_\epsilon(\chi)\le1$, since otherwise there is nothing to prove.
Now, using the periodicity of $\tilde{\chi}$, we first choose $m\in \N$ minimal such that $\mu_{m+1}=(M \mu)^{m+1} \mu_1<\frac{1}{2}$.
This implies that all $f_{ij,m}=\chi_{b_{ij}, \mu, \mu_{m+1}}f_{ij}$ vanish, since the support of the truncated cone requires that $k\in \Z^3\setminus\{0\}$ with $|k|\leq 2\mu_{m+1}<1$. Therefore, there exist constants $c_j\in \R$ such that $\tilde{\chi}_{jj}-c_j=\sigma_{jj,m}$ and hence by Lemma \ref{lem:inductive_step} it holds that
\begin{align}
\label{eq:energy_est1}
\begin{split}
\|\tilde{\chi}_{jj}-c_j\|_{L^2}^{2} 
&\leq C_{\gamma}^{m} \max\{ (\mu^{-2} + \mu_1^{-1} \epsilon^{-1}) E_{\epsilon}(\chi),(\mu^{-2} + \mu_1^{-1} \epsilon^{-1})^{(1-\gamma)^{m}} E_{\epsilon}(\chi)^{(1-\gamma)^{m}} \}\\
&\leq C_{\gamma}^{m}(\mu^{-2} + \mu_1^{-1} \epsilon^{-1})^{(1-\gamma)^{m}} E_{\epsilon}(\chi)^{(1-\gamma)^{m}}.
\end{split}
\end{align}
Here $C_{\gamma} = \frac{(C d)^{2d}}{\gamma^{36d}}$.
We hence choose $m$ (up to an integer cut-off) such that
\begin{align*}
m \sim \frac{\log(\mu_1)}{|\log(M\mu)|}.
\end{align*}
Recalling that $C_{\gamma}=\frac{(C d)^{2d}}{\gamma^{36d}}$ and taking into account the maximum on the right hand side of the above energy estimate, we choose $\gamma \in (0,1)$ such that
\begin{align*}
(1-\gamma)^m = c \in (1/2,3/4).
\end{align*}
To this end, we set
\begin{align*}
\gamma \sim \frac{|\log(c)|}{m}.
\end{align*}
As a consequence, \eqref{eq:energy_est1} turns into
\begin{align}
\label{eq:energy_est2}
\begin{split}
\|\tilde{\chi}_{jj}-c_j\|_{L^2}^2 
&\leq \tilde{C}^{m+1} e^{\tilde{C} m \log(m)}  (\mu^{-2} + \mu_1^{-1} \epsilon^{-1}) E_{\epsilon}(\chi)^{\frac{1}{2}},
\end{split}
\end{align}
where we used that $0\leq E_{\epsilon}(\chi)\leq 1$ and thus $E_{\epsilon}(\chi)^c\leq
  E_{\epsilon}(\chi)^{\frac{1}{2}}$.
Next, we optimize the prefactor $\mu^{-2} + \mu_1^{-1} \epsilon^{-1}$ by setting $\mu_1 = \mu \epsilon^{-1}>1$, which, by definition of $m$, yields that
\begin{align*}
m \sim \frac{\log(\mu)-\log(\epsilon)}{|\log(\mu M)|}.
\end{align*}
Choosing $\mu \sim \epsilon^{\alpha}$ for some $\alpha \in (0,1)$, \eqref{eq:energy_est2} becomes
\begin{align}
\label{eq:energy_est3}
\begin{split}
\|\tilde{\chi}_{jj}-c_j\|_{L^2}^2 
&\leq  \exp\left({\tilde{C} \frac{(1-\alpha)|\log(\epsilon)|}{\alpha |\log(\epsilon \tilde{M})|} \log\left( \frac{(1-\alpha)|\log(\epsilon)|}{\alpha|\log(\epsilon \tilde{M})|} \right)}\right) \exp(2\alpha |\log(\epsilon)|)  E_{\epsilon}(\chi)^{\frac{1}{2}} ,
\end{split}
\end{align}
where $\tilde{M} = M^{\frac{1}{\alpha}}$.
It remains to choose $\alpha \in (0,1)$ in order to balance the two exponential contributions involving terms of $|\log(\epsilon)|$. To this end, we set $\alpha = |\log(\epsilon)|^{-\frac{1}{2}}$. Hence, for $\epsilon_0>0$ sufficiently small and $\bar{C}>0$ a sufficiently large constant (independent of $\epsilon_0, \epsilon$), \eqref{eq:energy_est3} is bounded by
\begin{align}
\label{eq:energy_est4}
\begin{split}
\|\tilde{\chi}_{jj}-c_j\|_{L^2}^2 
&\leq  \exp\left(\bar{C} |\log(\epsilon)|^{\frac{1}{2}} \log\left(|\log(\epsilon)| \right) \right) E_{\epsilon}(u,\chi)^{\frac{1}{2}}  .
\end{split}
\end{align}
In order to conclude the argument, it hence suffices to estimate the left hand side. To this end, we note that by definition of the function $\tilde{\chi}_{jj}$ and by optimality of the mean,
\begin{align}
\label{eq:mean_est}
\|\tilde{\chi}_{jj}-c_j\|_{L^2}^2  \geq \|\chi_{jj}-\langle \chi_{jj} \rangle\|_{L^2}^2, 
\end{align}
where $\langle \chi_{jj} \rangle :=  \int\limits_{\T^3} \chi_{jj} dx$ denotes the mean value of $\chi_{jj}$. Moreover, we note that
\begin{align}
\label{eq:mean_cor}
\begin{split}
\|e(u) - \bar{e}\|_{L^2}
&\leq \|e(u) - \chi\|_{L^2} + \|\chi - \bar{e}\|_{L^2}
\leq E_{el}(u,\chi)^{\frac{1}{2}} +  |\bar{e} - \langle \chi \rangle|
+ \|\chi - \langle \chi \rangle\|_{L^2}\\
&\leq 2E_{el}(u,\chi)^{\frac{1}{2}} +  \|\chi - \langle \chi \rangle\|_{L^2},
\end{split}
\end{align}
where we have used Jensen's inequality in the last estimate. Combining \eqref{eq:mean_est} and \eqref{eq:mean_cor} and using the bound \eqref{eq:energy_est4}, then implies that
\begin{align*}
\|e(u) - \bar{e}\|_{L^2} 
&\leq 2E_{el}(u,\chi)^{\frac{1}{2}}  + \exp\left(\bar{C} |\log(\epsilon)|^{\frac{1}{2}} \log\left(|\log(\epsilon)| \right) \right) E_{\epsilon}(\chi)^{\frac{1}{4}}\\
&\leq 3 \exp\left(\bar{C} |\log(\epsilon)|^{\frac{1}{2}} \log\left(|\log(\epsilon)| \right) \right) E_{\epsilon}(\chi)^{\frac{1}{4}}.
\end{align*}
Taking this to the power four yields the desired result.
\end{proof}

\subsubsection{Proof of the lower bound in Theorem \ref{thm:scaling}}

Building on the proof of Theorem \ref{thm:rigidity}, we next turn to the proof of the lower scaling bound in Theorem \ref{thm:scaling}.

\begin{proof}[Proof of the lower bound in Theorem \ref{thm:scaling}]
The lower bound in Theorem \ref{thm:scaling} follows as a direct consequence of
the bound from Theorem \ref{thm:rigidity}. Indeed, for any $\nu \in (0,1)$ there exists $\epsilon_0>0$ such that for $\epsilon \in (0,\epsilon_0)$
we have
\begin{align*}
\|\chi- \bar{e}\|_{L^2}^2
&\leq 2\|e(u)-\chi\|_{L^2}^2 + 2\|e(u)-\bar{e}\|_{L^2}^2\\
& \leq 2E_{el}(u,\chi) + 2\exp(c_{\nu} |\log(\epsilon)|^{\frac{1}{2} + \nu}) E_{\epsilon}(u,\chi)\\
&\leq 4\exp(c_{\nu} |\log(\epsilon)|^{\frac{1}{2} + \nu}) E_{\epsilon}(u,\chi)^{\frac{1}{2}}.
\end{align*}
Since $\chi \in K$, we have that $\|\chi - \bar{e}\|_{L^2} \geq \dist(K, \bar{e})>0 $. As a consequence, for $\epsilon \in (0,\epsilon_0)$
\begin{align*}
\dist^2(K, \bar{e}) \leq 2\exp(c_{\nu} |\log(\epsilon)|^{\frac{1}{2} + \nu}) E_{\epsilon}(u,\chi)^{\frac{1}{2}}.
\end{align*}
Taking the infimum over all $u \in \mathcal{A}_{\overline{e}}$, $\chi \in BV(\Omega, K)$, implies that for $\epsilon \in (0,\epsilon_0)$
\begin{align*}
\exp(-c_{\nu} |\log(\epsilon)|^{\frac{1}{2} + \nu}) \dist^4(K, \bar{e}) \leq E_{\epsilon}(\overline{e}).
\end{align*}
This concludes the argument.
\end{proof}

\subsection{Remarks on the setting with Dirichlet data}
\label{sec:Dirichlet}

Concluding the quantitative analysis of the $T_3$-structures from \cite{BFJK94}, we comment on the difference between the periodic setting and the setting of a general Lipschitz domain with Dirichlet data. In this context, the first observation is that in the Dirichlet case there exist data such that the determinedness properties which had been crucially used in the iterative support reduction arguments are not always satisfied. More precisely, a central step in our argument above was the existence of polynomials $P_{ij\ell}$ such that $\p_{\ell}^h\chi_{jj} = P_{ij\ell}(\p_{\ell}^h\chi_{ii})$ for all $h \in \R$. We claim that in the Dirichlet setting there are data such that this is not always guaranteed, when considering solutions in the admissible class
\begin{align*}
\mathcal{A}^D_{\bar{e}}:= \{u \in H^1_{loc}(\R^3, \R^3): \ u(x) = (\bar{e} + s) x + b \mbox{ for } x \in \R^3 \setminus \Omega \},
\end{align*}
for some $s \in \mbox{skew}(3)$ and $b \in \R^3$ and $\bar{e}\in\mathcal{E}$, with $\mathcal{E}$ defined as in \eqref{eq:data-sets}. 
Indeed, setting $\tilde{\chi}_{ii} = \bar{e}_{ii}$ in $\R^3 \setminus \Omega$,
for any $v\in\mathbb{S}^2$, $h>0$ we have
\begin{align}
  \label{eq:determinedvalues}
\p_k^h\tilde\chi_{ii}\in\{0,\pm(\eta_1-\eta_2),\pm(\eta_1-\eta_3),\pm(\eta_2-\eta_3),\pm(\eta_1-\bar e_{ii}),\pm(\eta_2-\bar e_{ii}),\pm(\eta_3-\bar e_{ii})\}.
\end{align}
We remark that the last six values only arise in the Dirichlet and not in the periodic setting.
In order to have well-defined polynomials such that
\[
P_{ij k}(\p_k^h\tilde\chi_{ii})=\p_k^h\tilde\chi_{jj},
\]
all the thirteen values from \eqref{eq:determinedvalues} that $\p_k^h\tilde\chi_{ii}$ can attain, should be different.

\begin{example}
\label{ex:counter}
Let $\mathcal{E}$ be defined as in \eqref{eq:data-sets}.
We claim that there are boundary values $\bar{e} \in \mathcal{E}$ 
such that the set from \eqref{eq:determinedvalues} contains less than the maximal thirteen values. In order to observe this, consider $\bar e:=\lambda_1 J_1+\lambda_2 J_2+\lambda_3 J_3$ with $\lambda_1=\frac{1}{2}$ and $\lambda_2=\lambda_3=\frac{1}{4}$.
Then $\bar e_{11}=\frac{1}{2}\kappa+\frac{1}{2}\eta_1=\frac{1}{2}(\eta_2+\eta_3)$.
The admissible choice of parameters $\eta_2=\frac{1}{3}$, $\eta_1=1$ and $\eta_3=3$ leads to
\[
\eta_1-\eta_2=\frac{2}{3}=\bar e_{11}-\eta_1.
\]
The polynomial $P_{12 k}$ should be such that
\[
P_{12 k}(\eta_1-\eta_2)=\eta_2-\eta_3, \quad P_{12k}(\bar e_{11}-\eta_1)=\bar e_{22}-\eta_2,
\]
since $\p_k^h\chi_{11}=\eta_1-\eta_2$ when $\chi(x)=e^{(2)}$ and $\chi(x+h e_k)=e^{(3)}$ and $\p_k^h\chi_{11}=\bar e_{11}-\eta_1$ when $\chi(x)=e^{(3)}$ and $\chi(x+h e_k)=\bar e$.

But, by the choices of parameters above, we have
\[
\bar e_{22}=\Big(\frac{1}{2} J_1+\frac{1}{4} J_2+\frac{1}{4} J_3\Big)_{22}=\frac{1}{2} \eta_1+\frac{1}{4}\kappa+\frac{1}{4}\eta_1=\frac{13}{3},
\]
hence,
\[
\eta_2-\eta_3=-\frac{8}{3}\neq 4=\frac{13}{3}-\frac{1}{3}=\bar e_{22}-\eta_2
\]
so in this case $P_{12k}$ is not well-defined.
\end{example}

We remark that while counterexamples as in Example \ref{ex:counter} exist for certain choices of
$\overline{e}$, generically the values in \eqref{eq:determinedvalues} are
distinct.

\begin{lem}
  \label{lem:determinedE}
There exists an exceptional set 
\begin{align*}
E \subset \R^{3 \times 3}_{diag,tr}:=\{M \in \R^{3\times 3}: \ M \mbox{ is diagonal and } \tr(M)=\eta_1+\eta_2+\eta_3\} \simeq
\R^2,
\end{align*}
such that the set from \eqref{eq:determinedvalues} consists of thirteen distinct elements if $\bar{e} \in \mathcal{E}\setminus E$. 
The set $E$ consists of a finite union of lines and is hence a Lebesgue null set
in $\R^{3 \times 3}_{diag,tr}$.
\end{lem}

\begin{proof}
  In our proof of Lemma \ref{lem:determinedE}, we consider the set from
  \eqref{eq:determinedvalues}
  \begin{align*}
  D:= \{0,\pm(\eta_1-\eta_2),\pm(\eta_1-\eta_3),\pm(\eta_2-\eta_3),\pm(\eta_1-\bar e_{ii}),\pm(\eta_2-\bar e_{ii}),\pm(\eta_3-\bar e_{ii})\}
  \end{align*}
as parametrized by $\overline{e} \in \R^{3\times
  3}_{sym, tr}\simeq \R^2$ for given $\eta_j$.
Our main aim in the following is to show that all $13$ values in the definition of $D$ are distinct unless
$\overline{e}$ lies in a small exceptional set $E$.
Indeed, we claim that the set $E$ for which $D$ has less than $13$ elements is of the form
\begin{align}
  \label{eq:exceptionalset}
  E=\big\{\overline{e} \in \R^{3\times 3}_{diag, tr}\simeq \R^2 : \overline{e}_{ii} \in \Lambda \mbox{ for some } i\in\{1,2,3\}\big\},
\end{align}
where $\Lambda\subset \R$ is a finite set.
Once \eqref{eq:exceptionalset} is proved, the claim of the lemma follows: Indeed, viewed as a subset of $\R^3$ (diagonal matrices), the set $E$ thus is contained
inside a finite union of axis-parallel hyperplanes.
Moreover, the trace constraint $\text{tr}(\overline{e})=2\eta_1+\kappa$ corresponds to a non-axis-parallel
hyperplane, which intersects these hyperplanes along lines. Hence, the set $E$
consists of a finite number of lines, while $\bar e$ can be chosen from a
two-dimensional (relatively) open set. Therefore, the exceptional set is a
Lebesgue null set and generically $D$ contains $13$ values.

We provide the proof of \eqref{eq:exceptionalset}. To this end, we first consider the subset $D_1 \subset D$ of values generated by differences of the wells:
  \begin{align}
  \label{eq:welldifferences}
  D_1:=\{0,\pm(\eta_1-\eta_2),\pm(\eta_1-\eta_3),\pm(\eta_2-\eta_3)\}.
\end{align}
These seven values are distinct due to the ellipticity condition
\eqref{eq:elliptic} and are independent of $\overline{e}$.
We next show that the set $D_2$ of differences between $\overline{e}$ and the $i$-th component of the wells
$e^{(j)}$, $j\in\{1,2,3\}$,
\begin{align*}
  D_2:=\{\pm(\eta_1-\bar e_{ii}),\pm(\eta_2-\bar e_{ii}),\pm(\eta_3-\bar e_{ii})\},
\end{align*}
generically consists of six distinct values and that $D_2$ and $D_1$ are disjoint.

Concerning the distinct values, we note that
\begin{align*}
  \eta_j-\bar e_{ii}= \eta_k-\bar e_{ii} \Leftrightarrow \eta_j = \eta_k
\end{align*}
and all $\eta_{j}$ are distinct by assumption.
For the case of opposite signs, we instead observe that
\begin{align*}
   \eta_j-\bar e_{ii}= -(\eta_k-\bar e_{ii}) \Leftrightarrow \bar e_{ii} = \frac{\eta_{j}+\eta_{k}}{2},
\end{align*}
which is of the form claimed in \eqref{eq:exceptionalset}.

It remains to show that generically the sets $D_2$ and $D_1$ are disjoint.
Indeed, the set where $D_1\cap D_2 \neq \emptyset$ can be explicitly computed as
\begin{align*}
  \begin{split}
  \{\overline{e} \in \R^{3\times 3}_{sym, tr}\simeq \R^2 : \overline{e}_{ii} - \eta_l \in D_1 \text{ for some } i,l \in \{1,2,3\}\}
  &=\bigcup_{\lambda \in \tilde{\Lambda} } \{\overline{e} \in \R^{3\times 3}_{diag, tr}\simeq \R^2 :\overline{e}_{ii} = \lambda \},
  \end{split}
\end{align*}
where $\tilde{\Lambda}= \bigcup_{l} \eta_l +D_1$. This set is of the form
\eqref{eq:exceptionalset}, which concludes the proof.
\end{proof}

In the discussion of the quantitative lower bounds in the previous sections, we avoided the additional constraint encoded in Lemma \ref{lem:determinedE} by considering the
periodic setting, in which only the set from \eqref{eq:welldifferences} has to be considered. By the ellipticity assumption this always consists of disjoint elements.

The preceding observations from Lemma \ref{lem:determinedE} show that Theorems \ref{thm:rigidity} and \ref{thm:scaling} have direct analogues also in the Dirichlet setting in case that one considers Dirichlet data outside of the exceptional set $E$.
Indeed, by arguing exactly as in the above argument and using in addition the low frequency estimate in the final step in which the support is reduced to the zero frequency from \cite[Proposition 3.1(i)]{RRTT23}, for instance, we infer the following result:

\begin{thm}[Rigidity]
\label{thm:rigidity_Dirichlet}
Let $\bar{e} \in \R^{3\times 3}_{sym}$ be such that $\bar{e}\in \mathcal{E}\setminus E$. 
Let $E_{\epsilon}(u,\chi)$ be as in \eqref{eq:energy_total} with $\Omega \subset \R^3$ an arbitrary bounded Lipschitz domain and assume that $u \in \mathcal{A}_{\bar{e}}^D$ with
\begin{align}
\label{eq:admissible1}
\begin{split}
\mathcal{A}^D_{\bar{e}} &:= \{u \in W^{1,2}_{loc}(\R^3, \R^3): \ u(x) = (\bar{e} + s) x + b \mbox{ for } x \in \R^3 \setminus \Omega,\\
& \qquad \mbox{ for some } s \in \mbox{skew}(3), b \in \R^3\}.
\end{split}
\end{align}
Then for all $\nu \in (0,1)$ there exist $\epsilon_0>0$ and $c_{\nu}>0$ such that for all $\epsilon \in (0,\epsilon_0)$ we have that
\begin{align*}
\|e(u)- \bar{e}\|^4_{L^2(\Omega)} \leq \exp(c_{\nu}|\log(\epsilon)|^{\frac{1}{2} + \nu}) E_{\epsilon}(u,\chi).
\end{align*}
\end{thm}

\begin{proof}
The proof of Theorem \ref{thm:rigidity_Dirichlet} follows analogously to its periodic counterpart. The main difference consists of working with the continuous instead of the discrete Fourier transform and in the final optimization step using a low frequency estimate as, e.g., in \cite[Proposition 3.1.]{RRTT23}.
\end{proof}

\section{Proof of the upper bound in Theorem \ref{thm:scaling}}
\label{sec:upper}

In this section, we briefly discuss the upper bound construction in Theorem \ref{thm:scaling}. Since this is similar to the works \cite{C99,W97,RT22,RRT22} and also the qualitative constructions in \cite{BFJK94,M1}, we are rather brief at this point.

As the construction for the setting of prescribed affine Dirichlet data implies the construction for the case of prescribed mean value, we focus here on the Dirichlet setting. To this end, in what follows below, we will use the convention that $\Omega:=[0,1]^3$.

\begin{proof}[Proof of the upper bound in Theorem \ref{thm:scaling}]
As with at most four additional lamination steps, it is possible to reduce to this situation, without loss of generality, we assume that $\bar{e} = J_1:=\lambda e^{(1)} + (1-\lambda) J_3$ with $\lambda = \frac{\eta_1 + \eta_3 - 2\eta_2}{\eta_3 - \eta_1}$. We then argue iteratively. 

\emph{Step 1: First iteration.}
We realize the boundary conditions by a simple laminate between $e^{(1)}$ and $J_3$ with normal direction $b_{13}$ up to a cut-off function on a length scale $r_1 = \frac{1}{n}$ for a suitable choice of $n\in \N$:
\begin{align*}
u^{(1)}(x) = \big(u_1(r^{-1}_1 (x \cdot b_{13})) + \bar{e}x\big) (1-\dist(x/r_1, \partial \Omega))_+.
\end{align*}
Here $u_1$ denotes a displacement such that
\begin{align*}
e(u^{(1)}) = 
\left\{ \begin{array}{ll}
e^{(1)}, & x\cdot b_{13} \in [0, \lambda]\\
J_3,& x\cdot b_{13} \in [\lambda, 1],
\end{array} \right.
\end{align*}
which is periodically extended, and $\dist(x,\partial \Omega)$ is the distance function close to $\partial \Omega$ and
constant for $\dist(x,\partial \Omega)\geq \frac{1}{4}$.
We further define $\chi^{(1)}$ to be the pointwise projection (wherever it is well-defined and $e^{(1)}$ else) onto $K_3$. We then obtain
\begin{align*}
&E_{el}(u^{(1)},\chi^{(1)}) \leq C \left( (1-\lambda) + r_1 \right),\
E_{surf}(u^{(1)},\chi^{(1)}) \leq \frac{C}{r_1}, \ E_{\epsilon}(u^{(1)},\chi^{(1)}) \leq C\left((1-\lambda) + \frac{1}{r_1}\right).
\end{align*}

\emph{Step 2: Iterative construction.}
We iteratively increase the volume fractions in which $e(u^{(j)})$ is in the wells. 
We assume that at step $j-1$ (up to cut-off regions) $u^{(j-1)}$ is a nested laminate consisting of nested stripes of length scales $r_1,\dots,r_{j-1}$ with  $r_{\ell} \in (0, r_{\ell-1}/2)$ for all $\ell \in \N$, 
$e(u^{(j-1)}) \in K_3$ or $e(u^{(j-1)}) = J_{i}$ with the domains in which $e(u^{(j-1)}) = J_{i}$ consisting of stripes of width $r_{j-1}$ intersected with $\Omega$, and that the phase indicator $\chi^{(j-1)}$ is given by a pointwise projection onto $K_3$. Further, we assume that this carries an energy contribution of the form
\begin{align*}
E_{el}(u^{(j-1)},\chi^{(j-1)}) \leq C \left( \lambda^{j} + \sum\limits_{\ell = 2}^{j-1} \lambda^{\ell} \frac{r_{\ell}}{r_{\ell-1}} \right),
\ E_{surf}(u^{(j-1)},\chi^{(j-1)}) \leq C \frac{1}{r_{j-1}}.
\end{align*}
We then define $u^{(j)}, \chi^{(j)}$ by using that
\begin{align*}
J_{i} = \lambda e^{(i)} + (1-\lambda) J_{i-1},
\end{align*}
where we consider all indices modulo three, in particular, $e^{(0)} = e^{(3)}$, $J_0 = J_{3}$ and where $\lambda \in (0,1)$ is as above. More precisely, in any stripe intersected with $\Omega$ outside of the cut-off regions in which $e(u^{(j-1)}) = J_{i}$, we replace $u^{(j-1)}$ by a simple laminate (up to cut-offs) of periodicity $r_j$. The cut-off $ \chi^{(j)}$ is defined by the pointwise projection onto the wells. This improves the volume fraction in which the wells are attained by a factor $\lambda$ and leads to new surface energy contributions of cost $C/r_{j}$. Hence,
\begin{align*}
E_{el}(u^{(j)},\chi^{(j)}) \leq C \left( \lambda^{j} + \sum\limits_{\ell = 2}^{j} \lambda^{\ell} \frac{r_{\ell}}{r_{\ell-1}} \right),
\ E_{surf}(u^{(j)},\chi^{(j)}) \leq C \frac{1}{r_j}.
\end{align*}

\emph{Step 3: Energy estimates and optimization.}
In order to deduce the upper bound, it remains to estimate the overall energy contribution. To this end, we choose $r_j = r^j$ for some $r \in (0,1)$ to be fixed. Thus,
\begin{align*}
E_{\epsilon}(u^{(j)},\chi^{(j)}) \leq C\left( \lambda^j + r + \epsilon r^{-j} \right).
\end{align*}
Optimizing in $r$, we set $r \sim \epsilon^{\frac{1}{j+1}}$ and $j \sim \sqrt{\left|\tfrac{\log(\epsilon)}{\log(\lambda)}\right|}$. This then proves the claim with a constant $C=C(\lambda)>0$.
\end{proof}

\begin{rmk}
\label{rmk:exp}
Foreshadowing our numerical findings, the above construction shows that the length scales of the involved microstructure decrease exponentially with the lamination order $j \in \N$. We hence expect that a high resolution of this in three dimensions is extremely costly.
\end{rmk}

\part{Numerical Simulations of the \texorpdfstring{$T_3$}{T3} Microstructure}
\label{part:3}

\section{Numerical energy minimization problem}

The numerical simulations of the $T_3$ microstructures assume a material with three wells characterized by the eigenstrains
\begin{align*}
    e^{(1)}= \begin{pmatrix}
    \eta_1 & 0 & 0\\
    0 & \eta_2 & 0\\
    0 & 0 & \eta_3
\end{pmatrix}, \quad 
e^{(2)}= \begin{pmatrix}
    \eta_2 & 0 & 0\\
    0 & \eta_3 & 0\\
    0 & 0 & \eta_1
\end{pmatrix},\quad
e^{(3)}= \begin{pmatrix}
    \eta_3 & 0 & 0\\
    0 & \eta_1 & 0\\
    0 & 0 & \eta_2
\end{pmatrix},
\end{align*}
where $\eta_1,\eta_2,\eta_3\in \R$ are subject to the constraints~\cite{BhattacharyaEtAl1994}
\begin{align*}
    \eta_2 < \eta_1 < \eta_3\quad \text{and} \quad \eta_2+\eta_3 > 2 \eta_1.
\end{align*}
Here, unlike in the previous parts, we follow the numbering of the eigenstrains from~\cite{BhattacharyaEtAl1994}.
The above eigenstrains are typical for shape-memory alloys that undergo a cubic-to-orthorhombic phase transformation with three low-symmetry variants at room temperature. 

We assume a linear elastic material within the setting of linearized kinematics, so that each variant $\alpha=1,2,3$ is described by an elastic strain energy density
\begin{align*}
    E^{(\alpha)}(e) := \frac{1}{2}\left(e_{ij}-e^{(\alpha)}_{ij}\right)\C_{ijkl}\left(e_{kl}-e^{(\alpha)}_{kl}
    \right),
\end{align*}
where $\C$ is the fourth-order elasticity tensor. For simplicity, we assume elastic isotropy such that
\begin{align*}
    \C_{ijkl} = \lambda_e \delta_{ij}\delta_{kl} + \mu_e \left(\delta_{ik}\delta_{jl}+\delta_{il}\delta_{jk}\right)
\end{align*}
with Lam\'e constants $\lambda_e$ and $\mu_e$. All material and simulation parameters are listed in Tab.~\ref{tab:MaterialParams}.

To simulate complex microstructures, we reformulate the energy minimization problem as one that is numerically more readily tractable. To this end, we assume that at all three variants may in principle coexist, having volume fractions $\chi=(\chi^{(1)},\chi^{(2)},\chi^{(3)})$ with $\chi^{(\alpha)}\geq 0$ and $\sum_{\alpha=1}^3 \chi^{(\alpha)}=1$. Utilizing a Taylor model, the rule of mixture yields the total elastic energy density as
\begin{align*}
    E(e,\chi): = \sum^3_{\alpha=1} \chi^{(\alpha)} E^{(\alpha)}(e).
\end{align*}
Akin to phase-field models, the volume fractions $\chi^{(\alpha)}$ in this model act as phase fields, yet lacking a regularization. As we are interested in energy-minimizing microstructures over a domain $\Omega$, we must introduce a regularization such as to obtain physically meaningful microstructural patterns. In the phase field context, this is typically accomplished by introducing an interfacial energy. Here, instead, we introduce a configurational (not physical) entropy \cite{TanKochmann2017,KumarEtAl2020}
\begin{align*}
    S(\chi):=-k_T\sum^3_{\alpha=1}\chi^{(\alpha)}\log\chi^{(\alpha)},
\end{align*}
where $k_T>0$ is a constant (in physical models related to temperature). As a consequence, the free energy density can be defined as
\begin{align}
  \label{eq:freeEnergyF0}
    E(e,\chi):= W(e,\chi) - S(\chi) = \sum^3_{\alpha=1} \chi^{(\alpha)} \left(E^{(\alpha)}(e) + k_T\log\chi^{(\alpha)}\right).
\end{align}
Minimizing this free energy density allows to determine the energy-minimizing volume fractions locally as \cite{KumarEtAl2020}
\begin{align}
  \label{eq:lambdaStar}
    \chi_*^{(\alpha)}(e) := \argmin E(e,\cdot) = \frac{\exp\left(-\frac{E^{(\alpha)}(e)}{k_T}\right)}{\sum^3_{\beta=1}\exp\left(-\frac{E^{(\beta)}(e)}{k_T}\right)},
\end{align}
which, when inserted into~\eqref{eq:freeEnergyF0}, yields the effective potential energy density
\begin{align}
  \label{eq:condensedF0}
    E_*(e):=E\left(e,\chi_*(e)\right)=-k_T\log\left(\sum^3_{\alpha=1}\exp\left(-\frac{E^{(\alpha)}(e)}{k_T}\right)\right).
\end{align}
This is a numerically convenient approximation of 
\begin{align}
  \label{eq:ExactE}
    E(e):=\min\left\{E^{(1)}(e),E^{(2)}(e),E^{(3)}(e) \right\}, 
\end{align}
in which the constant $k_T>0$ can be used to control the ``sharpness'' of results. In the limit $k_T\to 0$, \eqref{eq:condensedF0} approaches \eqref{eq:ExactE}~\cite{TanKochmann2017,KumarEtAl2020}. Therefore, for low values of $k_T$, $\chi_*^{(\alpha)}$ approaches $0$ or $1$ for almost every choice of $e$, so that minimizing the total potential energy
\begin{align}
  \label{eq:totEng}
    E_{el}(u) := \int_{\Omega} E_*(e)
\end{align}
produces sharp patterns, in which every material point is (approximately) associated with a single variant $\alpha$ through $\chi^{(\alpha)}\approx 1$. Increasing $k_T$ results in more diffuse solutions, admitting intermediate values of $\chi_*^{(\alpha)}(e)\in(0,1)$ and hence resulting in patterns with more  diffuse interfaces. For our purposes, we will choose the former case and aim for sharp microstructures (with the value of $k_T$ in Tab.~\ref{tab:MaterialParams}).

From \eqref{eq:condensedF0} we obtain the Cauchy stress tensor as
\begin{align}
  \label{eq:stressStar}
    \sigma=\frac{\mathrm{d}E_*}{\mathrm{d}e}=\frac{\partial E_*}{\partial e} + \frac{\partial E_*}{\partial\chi_*}\frac{\mathrm{d}\chi_*}{\mathrm{d}e}
 =\sum^3_{\alpha=1}\chi_*^{(\alpha)}\C\left(e-e^{(\alpha)}\right),
\end{align}
where $\sigma^{(\alpha)}:=\C\left(e-e^{(\alpha)}\right)$ is the stress tensor associated with variant~$\alpha$. Note that $\partial E_*/\partial\chi_*=0$ by the definition of $\chi_*$ in \eqref{eq:lambdaStar}.

Finally, rendering \eqref{eq:totEng} stationary over a domain $\Omega$ is equivalent to solving linear momentum balance
\begin{align}
  \label{eqlinMomBalance0}
    \nabla\cdot\sigma =0\quad \text{in }\Omega,
\end{align}
with suitable boundary conditions.

\begin{table}[tb]
	\centering
	\caption{Material constants and simulation parameters.}
	\begin{tabular}{ll}
		\toprule
		parameter  & value  \\
		\midrule
		$\eta_1$  & 0.03  \\
        $\eta_2$  & 0.01  \\
        $\eta_3$  & 0.06  \\
        $\kappa$  & 0.04  \\
        $\lambda_{\text{e}}$ & 7/3\\
        $\mu_{\text{e}}$  & 1 \\
		$k_T$ & 10$^{-4}$    \\
        $L$ & 1    \\
        $N$ & 128, 256    \\
		\bottomrule
	\end{tabular}
	\label{tab:MaterialParams}
\end{table}

\section{Spectral Homogenization}

To identify energy-minimizing microstructures, we solve~\eqref{eqlinMomBalance0} numerically on a cube-shaped representative volume element (RVE) $\Omega=[0,L]^3$ of side length $L$. Without loss of generality, we apply periodic boundary conditions, which motivates the use of a Fourier-based spectral solution scheme. We discretize the cubic RVE into a regular grid of $N\times N\times N$ nodes. We here follow~\cite{Moulinec1998,Moulinec2003} and solve~\eqref{eqlinMomBalance0} in Fourier space. To this end, we define the (discrete) inverse Fourier transform of the displacement field $u$ (and any other function analogously) as
\begin{align}
  \label{eq:FourierTransform0}
	u(x) = \mathcal{F}^{-1}\left(\hat u\right)
		= \sum_{k\in\mathcal{T}} \hat u(k)\,\exp\left(i\,k\cdot x\right) \st{and} i =\sqrt{-1},
\end{align}
where $\hat u(k)$ are the Fourier coefficient vectors, and $k$ are wave vectors in the reciprocal lattice $\mathcal{T}$. Note that, since $\sum^3_{\alpha=1}\chi_*^{(\alpha)}=1$, 
\begin{align*}
    \sigma = \sum^3_{\alpha=1}\chi_*^{(\alpha)}\C\left(e-e^{(\alpha)}\right)
    = \C\, e - \C \sum^3_{\alpha=1}\chi_*^{(\alpha)} e^{(\alpha)} 
    =: \C\,e - \tau ,
\end{align*}
so that we must solve 
\begin{align}
  \label{eq:FourierLMB}
    \nabla \cdot\left(\C\,e - \tau\right) = 0.
\end{align}
Inserting \eqref{eq:FourierTransform0} into \eqref{eq:FourierLMB} along with the strain-displacement relation $e=\mathrm{sym}(\nabla u)$ and exploiting minor symmetry of $\C$ yields
\begin{align*}
    \C_{ijkl}u_{k,lj} = \tau_{ij,j}
    \RA
    -\C_{ijkl}\hat u_{k}(k)k_lk_j = i k_j\hat\tau_{ij}(k),
\end{align*}
which is an algebraic equation to be solved for the Fourier coefficients $\hat u_{k}$. Defining the acoustic tensor $A_{ik}(k) = \C_{ijkl} k_j k_l$ (which is positive-definite for $k\neq 0$), the strain-displacement relation finally leads to the solution
\begin{align}
  \label{eq:VarepsFourier0}
	\hat e_{ij}(k)
	= \begin{cases}
		\cfrac{1}{2} 
			\left[A_{ki}^{-1}(\bfk)k_j + A_{kj}^{-1}(\bfk) k_i \right]
			\hat\tau_{kl}(\bfk)k_l ,& k\neq 0,\\[0.2cm]
		\quad \overline e_{ij} ,& k=0,
		\end{cases}
\end{align}
whose second case ($k=0$) allows us to impose the volume-averaged, macroscopic strain tensor $\overline e$. Unfortunately, $\tau$ depends on $\chi_*$, which in turn depends on $e$. Therefore, \eqref{eq:VarepsFourier0} is solved by fixed-point iteration, until convergence is achieved. That is, starting from an initial guess for $e(x)$, we compute $\tau(x)$ in real space on the discrete grid, apply a Fourier transform to obtain $\tau(k)$, then apply \eqref{eq:VarepsFourier0} to solve for $\hat e$, back-transform to real space via the inverse Fourier transform to obtain $e(x)$ and further $\tau(x)$, and we repeat this procedure till convergence.

\subsection*{Acknowledgements}
A.R.~and A.T.~gratefully acknowledge funding by the Deutsche Forschungsgemeinschaft (DFG, German Research Foundation) through SPP 2256, project ID 441068247. C.Z.~acknowledges funding by the Deutsche Forschungsgemeinschaft (DFG, German Research Foundation) – Project-ID 258734477 – SFB 1173. R.I.~and D.M.~Kochmann gratefully acknowledge the support from the Swiss National Science Foundation (SNSF) through project 200021-178747.
 
\bibliographystyle{alpha}
\bibliography{citations1_v31,DMK}

\newcommand{\etalchar}[1]{$^{#1}$}
\begin{thebibliography}{SCFHW15}

\bibitem[Bal90]{B90}
John~M Ball.
\newblock Sets of gradients with no rank-one connections.
\newblock {\em Journal de math{\'e}matiques pures et appliqu{\'e}es},
  69(3):241--259, 1990.

\bibitem[BFJK94a]{BFJK94}
Kaushik Bhattacharya, Nikan~B. Firoozye, Richard~D. James, and Robert~V. Kohn.
\newblock Restrictions on microstructure.
\newblock {\em Proceedings of the Royal Society of Edinburgh: Section A
  Mathematics}, 124(5):843–878, 1994.

\bibitem[BFJK94b]{BhattacharyaEtAl1994}
Kaushik Bhattacharya, Nikan~B. Firoozye, Richard~D. James, and Robert~V. Kohn.
\newblock Restrictions on microstructure.
\newblock {\em Proceedings of the Royal Society of Edinburgh: Section A
  Mathematics}, 124(5):843--878, 1994.

\bibitem[BM23]{BM23}
Peter Bella and Roberta Marziani.
\newblock {$\Gamma $}-convergence for plane to wrinkles transition problem.
\newblock {\em arXiv preprint arXiv:2312.06469}, 2023.

\bibitem[CDPR{\etalchar{+}}20]{CDPRZZ20}
Pierluigi Cesana, Francesco Della~Porta, Angkana R{\"u}land, Christian
  Zillinger, and Barbara Zwicknagl.
\newblock Exact constructions in the (non-linear) planar theory of elasticity:
  from elastic crystals to nematic elastomers.
\newblock {\em Archive for Rational Mechanics and Analysis}, 237(1):383--445,
  2020.

\bibitem[Chi99]{C99}
Michel Chipot.
\newblock The appearance of microstructures in problems with incompatible wells
  and their numerical approach.
\newblock {\em Numer. Math.}, 83(3):325--352, 1999.

\bibitem[CK02]{CK00}
Miroslav Chleb{\'\i}k and Bernd Kirchheim.
\newblock Rigidity for the four gradient problem.
\newblock {\em Journal f{ü}r die reine und angewandte Mathematik}, 551(551),
  2002.

\bibitem[CKZ17]{CKZ17}
Sergio Conti, Matthias Klar, and Barbara Zwicknagl.
\newblock Piecewise affine stress-free martensitic inclusions in planar
  nonlinear elasticity.
\newblock {\em Proceedings of the Royal Society A: Mathematical, Physical and
  Engineering Sciences}, 473(2203):20170235, 2017.

\bibitem[CO09]{CO1}
Antonio Capella and Felix Otto.
\newblock A rigidity result for a perturbation of the geometrically linear
  three-well problem.
\newblock {\em Communications on Pure and Applied Mathematics},
  62(12):1632--1669, 2009.

\bibitem[CO12]{CO}
Antonio Capella and Felix Otto.
\newblock A quantitative rigidity result for the cubic-to-tetragonal phase
  transition in the geometrically linear theory with interfacial energy.
\newblock {\em Proceedings of the Royal Society of Edinburgh: Section A
  Mathematics, 142 , pp 273-327 doi:10.1017/S0308210510000478}, 2012.

\bibitem[CS13]{CS13}
Isaac~Vikram Chenchiah and Anja Schl{\"o}merkemper.
\newblock Non-laminate microstructures in monoclinic-{I} martensite.
\newblock {\em Archive for Rational Mechanics and Analysis}, 207(1):39--74,
  2013.

\bibitem[CT93]{CT93}
E~Casadio-Tarabusi.
\newblock An algebraic characterization of quasi-convex functions.
\newblock {\em Ricerche Mat}, 42(1):11--24, 1993.

\bibitem[Dac07]{D}
Bernard Dacorogna.
\newblock {\em Direct methods in the calculus of variations}, volume~78.
\newblock Springer, 2007.

\bibitem[DM95a]{DM2}
Georg Dolzmann and Stefan M{\"u}ller.
\newblock The influence of surface energy on stress-free microstructures in
  shape memory alloys.
\newblock {\em Meccanica}, 30:527--539, 1995.
\newblock 10.1007/BF01557083.

\bibitem[DM95b]{DM1}
Georg Dolzmann and Stefan M{\"u}ller.
\newblock Microstructures with finite surface energy: the two-well problem.
\newblock {\em Archive for Rational Mechanics and Analysis}, 132:101--141,
  1995.

\bibitem[DPR20]{DPR20}
Francesco Della~Porta and Angkana R{\"u}land.
\newblock Convex integration solutions for the geometrically nonlinear two-well
  problem with higher {S}obolev regularity.
\newblock {\em Mathematical Models and Methods in Applied Sciences},
  30(03):611--651, 2020.

\bibitem[FS08]{FS08}
Daniel Faraco and L{\'a}szl{\'o} Sz{\'e}kelyhidi.
\newblock Tartar’s conjecture and localization of the quasiconvex hull in
  {$\mathbb{R}^{2\times 2}$}.
\newblock {\em Acta mathematica}, 200(2):279--305, 2008.

\bibitem[FSJ18]{FS17}
Clemens F{\"o}rster and L{\'a}szl{\'o} Sz{\'e}kelyhidi~Jr.
\newblock {$T_5$}-configurations and non-rigid sets of matrices.
\newblock {\em Calculus of Variations and Partial Differential Equations},
  57(1):19, 2018.

\bibitem[GN04]{GN04}
A~Garroni and V~Nesi.
\newblock Rigidity and lack of rigidity for solenoidal matrix fields.
\newblock {\em Proceedings of the Royal Society of London. Series A:
  Mathematical, Physical and Engineering Sciences}, 460(2046):1789--1806, 2004.

\bibitem[Gra08]{Grafakos}
Loukas Grafakos.
\newblock {\em Classical {F}ourier analysis}, volume~2.
\newblock Springer, 2008.

\bibitem[GRTZ24]{GRTZ24}
Janusz Ginster, Angkana R{\"u}land, Antonio Tribuzio, and Barbara Zwicknagl.
\newblock On the effect of geometry on scaling laws for a class of martensitic
  phase transformations.
\newblock {\em arXiv preprint arXiv:2405.05927}, 2024.

\bibitem[GZ24]{GZ24}
Janusz Ginster and Barbara Zwicknagl.
\newblock Energy scaling laws for microstructures: from helimagnets to
  martensites.
\newblock {\em Calculus of Variations and Partial Differential Equations},
  63(1):8, 2024.

\bibitem[Kir03]{K1}
Bernd Kirchheim.
\newblock Rigidity and geometry of microstructures.
\newblock {\em MPI-MIS lecture notes}, 2003.

\bibitem[KKO13]{KKO13}
Hans Kn{\"u}pfer, Robert~V Kohn, and Felix Otto.
\newblock Nucleation barriers for the cubic-to-tetragonal phase transformation.
\newblock {\em Communications on pure and applied mathematics}, 66(6):867--904,
  2013.

\bibitem[KLLR19a]{KLLR19b}
Georgy Kitavtsev, Gianluca Lauteri, Stephan Luckhaus, and Angkana R\"{u}land.
\newblock A compactness and structure result for a discrete multi-well problem
  with {$SO(n)$} symmetry in arbitrary dimension.
\newblock {\em Arch. Ration. Mech. Anal.}, 232(1):531--555, 2019.

\bibitem[KLLR19b]{KLLR19a}
Georgy Kitavtsev, Gianluca Lauteri, Stephan Luckhaus, and Angkana R{\"u}land.
\newblock A compactness and structure result for a discrete multi-well problem
  with {$\text{SO}(n)$} symmetry in arbitrary dimension.
\newblock {\em Archive for Rational Mechanics and Analysis}, 232:531--555,
  2019.

\bibitem[KM{\v{S}}03]{KMS03}
Bernd Kirchheim, Stefan M{\"u}ller, and Vladim{\'\i}r {\v{S}}ver{\'a}k.
\newblock Studying nonlinear {P}{D}{E} by geometry in matrix space.
\newblock In {\em Geometric analysis and nonlinear partial differential
  equations}, pages 347--395. Springer, 2003.

\bibitem[Koh91]{K91}
Robert~V. Kohn.
\newblock The relaxation of a double-well energy.
\newblock {\em Continuum Mechanics and Thermodynamics}, 3:193--236, 1991.

\bibitem[KVK20]{KumarEtAl2020}
Siddhant Kumar, A.~Vidyasagar, and Dennis~M. Kochmann.
\newblock An assessment of numerical techniques to find energy-minimizing
  microstructures associated with nonconvex potentials.
\newblock {\em International Journal for Numerical Methods in Engineering},
  121(7):1595--1628, 2020.

\bibitem[KW14]{KW14}
Robert~V Kohn and Benedikt Wirth.
\newblock Optimal fine-scale structures in compliance minimization for a
  uniaxial load.
\newblock {\em Proceedings of the Royal Society A: Mathematical, Physical and
  Engineering Sciences}, 470(2170):20140432, 2014.

\bibitem[KW16]{KW16}
Robert~V Kohn and Benedikt Wirth.
\newblock Optimal fine-scale structures in compliance minimization for a shear
  load.
\newblock {\em Communications on Pure and Applied Mathematics},
  69(8):1572--1610, 2016.

\bibitem[MS98]{Moulinec1998}
H.~Moulinec and P.~Suquet.
\newblock {A numerical method for computing the overall response of nonlinear
  composites with complex microstructure}.
\newblock {\em Computer Methods in Applied Mechanics and Engineering},
  157(1-2):69--94, 1998.

\bibitem[M{\v{S}}99]{MS}
Stefan M{\"u}ller and Vladim{\'i}r {\v{S}}ver{\'a}k.
\newblock Convex integration with constraints and applications to phase
  transitions and partial differential equations.
\newblock {\em Journal of the European Mathematical Society}, 1:393--422, 1999.
\newblock 10.1007/s100970050012.

\bibitem[MS03a]{Moulinec2003}
H.~Moulinec and P.~Suquet.
\newblock {Comparison of FFT-based methods for computing the response of
  composites with highly contrasted mechanical properties}.
\newblock {\em Physica B: Condensed Matter}, 338(1-4):58--60, 2003.

\bibitem[M{\v{S}}03b]{MS03}
Stefan M{\"u}ller and Vladimir {\v{S}}ver{\'a}k.
\newblock Convex integration for {L}ipschitz mappings and counterexamples to
  regularity.
\newblock {\em Annals of mathematics}, 157(3):715--742, 2003.

\bibitem[M{\"u}l99]{M1}
Stefan M{\"u}ller.
\newblock Variational models for microstructure and phase transitions.
\newblock In {\em Calculus of variations and geometric evolution problems},
  pages 85--210. Springer, 1999.

\bibitem[NM91]{NM91}
Vincenzo Nesi and Graeme~W Milton.
\newblock Polycrystalline configurations that maximize electrical resistivity.
\newblock {\em Journal of the Mechanics and Physics of Solids}, 39(4):525--542,
  1991.

\bibitem[Pom10]{P10}
Waldemar Pompe.
\newblock The quasiconvex hull for the five-gradient problem.
\newblock {\em Calculus of Variations and Partial Differential Equations},
  37(3):461--473, 2010.

\bibitem[PP04]{PP04}
Mariapia Palombaro and Marcello Ponsiglione.
\newblock The three divergence free matrix fields problem.
\newblock {\em Asymptotic Analysis}, 40(1):37--49, 2004.

\bibitem[PW22]{PW22}
Jonas Potthoff and Benedikt Wirth.
\newblock Optimal fine-scale structures in compliance minimization for a
  uniaxial load in three space dimensions.
\newblock {\em ESAIM: Control, Optimisation and Calculus of Variations}, 28:27,
  2022.

\bibitem[Rin18]{Ri18}
Filip Rindler.
\newblock {\em Calculus of variations}.
\newblock Springer, 2018.

\bibitem[RRT23]{RRT22}
Bogdan Rai{\c{t}}{\u{a}}, Angkana R{\"u}land, and Camillo Tissot.
\newblock On scaling properties for two-state problems and for a singularly
  perturbed {$T_3$} structure.
\newblock {\em Acta Applicandae Mathematicae}, 184(1):5, 2023.

\bibitem[RRTT24]{RRTT23}
Bogdan Rai{\c{t}}{\u{a}}, Angkana R{\"u}land, Camillo Tissot, and Antonio
  Tribuzio.
\newblock On scaling properties for a class of two-well problems for higher
  order homogeneous linear differential operators.
\newblock {\em SIAM Journal on Mathematical Analysis}, 56(3):3720--3758, 2024.

\bibitem[RS23]{RS23}
Angkana R{\"u}land and Theresa~M Simon.
\newblock On rigidity for the four-well problem arising in the
  cubic-to-trigonal phase transformation.
\newblock {\em Journal of Elasticity}, 153(3):455--475, 2023.

\bibitem[RT22]{RT22}
Angkana R{\"u}land and Antonio Tribuzio.
\newblock On the energy scaling behaviour of a singularly perturbed {T}artar
  square.
\newblock {\em Archive for Rational Mechanics and Analysis}, 243(1):401--431,
  2022.

\bibitem[RT23]{RT22a}
Angkana R{\"u}land and Antonio Tribuzio.
\newblock On the energy scaling behaviour of singular perturbation models with
  prescribed {D}irichlet data involving higher order laminates.
\newblock {\em ESAIM: Control, Optimisation and Calculus of Variations}, 29:68,
  2023.

\bibitem[RT24]{RT23}
Angkana R{\"u}land and Antonio Tribuzio.
\newblock On the scaling of the cubic-to-tetragonal phase transformation with
  displacement boundary conditions.
\newblock {\em Journal of Elasticity}, pages 1--39, 2024.

\bibitem[RTZ18]{RTZ19}
Angkana R{\"u}land, Jamie~M Taylor, and Christian Zillinger.
\newblock Convex integration arising in the modelling of shape-memory alloys:
  some remarks on rigidity, flexibility and some numerical implementations.
\newblock {\em Journal of Nonlinear Science}, pages 1--48, 2018.

\bibitem[R{\"u}l16]{R16}
Angkana R{\"u}land.
\newblock The cubic-to-orthorhombic phase transition: Rigidity and non-rigidity
  properties in the linear theory of elasticity.
\newblock {\em Archive for Rational Mechanics and Analysis}, 221(1):23--106,
  2016.

\bibitem[R{\"u}l22]{R22}
Angkana R{\"u}land.
\newblock Rigidity and flexibility in the modelling of shape-memory alloys.
\newblock {\em Research in Mathematics of Materials Science}, pages 501--515,
  2022.

\bibitem[RZZ19]{RZZ16}
Angkana R{\"u}land, Christian Zillinger, and Barbara Zwicknagl.
\newblock Higher {S}obolev regularity of convex integration solutions in
  elasticity: The planar geometrically linearized hexagonal-to-rhombic phase
  transformation.
\newblock {\em Journal of Elasticity,
  https://doi.org/10.1007/s10659-018-09719-3}, 2019.

\bibitem[SCFHW15]{CS15}
Anja Schl{\"o}merkemper, Isaac~V Chenchiah, Rainer Fechte-Heinen, and Daniel
  Wachsmuth.
\newblock Upper and lower bounds on the set of recoverable strains and on
  effective energies in cubic-to-monoclinic martensitic phase transformations.
\newblock In {\em MATEC Web of Conferences}, volume~33, page 02011. EDP
  Sciences, 2015.

\bibitem[Sch75]{Sch75}
Vladimir Scheffer.
\newblock Regularity and irregularity of solutions to nonlinear second-order
  elliptic systems of partial differential-equations and inequalities.
\newblock {\em Princeton University Dissertations}, 1975.

\bibitem[Sim21a]{S21a}
Theresa~M Simon.
\newblock Quantitative aspects of the rigidity of branching microstructures in
  shape memory alloys via h-measures.
\newblock {\em SIAM Journal on Mathematical Analysis}, 53(4):4537--4567, 2021.

\bibitem[Sim21b]{S21}
Theresa~M Simon.
\newblock Rigidity of branching microstructures in shape memory alloys.
\newblock {\em Archive for Rational Mechanics and Analysis}, 241(3):1707--1783,
  2021.

\bibitem[ST23]{TS23}
Massimo Sorella and Riccardo Tione.
\newblock The four-state problem and convex integration for linear differential
  operators.
\newblock {\em Journal of Functional Analysis}, 284(4):109785, 2023.

\bibitem[{\v{S}}ve93]{S93}
Vladim{\'\i}r {\v{S}}ver{\'a}k.
\newblock On {T}artar's conjecture.
\newblock In {\em Annales de l'IHP Analyse non lin{\'e}aire}, volume~10 of {\em
  4}, pages 405--412, 1993.

\bibitem[Tar93]{T93}
Luc Tartar.
\newblock Some remarks on separately convex functions.
\newblock In {\em Microstructure and phase transition}, pages 191--204.
  Springer, 1993.

\bibitem[TK17]{TanKochmann2017}
Wei~Lin Tan and Dennis~M. Kochmann.
\newblock An effective constitutive model for polycrystalline ferroelectric
  ceramics: Theoretical framework and numerical examples.
\newblock {\em Computational Materials Science}, 136:223--237, 2017.

\bibitem[\v{S}93]{Sverak}
Vladim\'{\i}r \v{S}ver\'{a}k.
\newblock On the problem of two wells.
\newblock In {\em Microstructure and phase transition}, volume~54 of {\em IMA
  Vol. Math. Appl.}, pages 183--189. Springer, New York, 1993.

\bibitem[Win97]{W97}
Matthias Winter.
\newblock An example of microstructure with multiple scales.
\newblock {\em European J. Appl. Math.}, 8(2):185--207, 1997.

\end{thebibliography}

\end{document}